\documentclass[11pt]{article} 

\usepackage{amssymb}
\usepackage{amsfonts, amsmath, amscd, amssymb,amsthm}
\usepackage{amsmath}
\usepackage{xcolor}
\usepackage{ulem}
\allowdisplaybreaks

\usepackage{geometry} 
\usepackage{graphicx} 
\usepackage{hyperref}

\hypersetup{pdfstartview=FitH,  linkcolor=blue,urlcolor=urlcolor, citecolor=blue, colorlinks=true}

\newtheorem{theorem}{Theorem}[section]
\newtheorem{lemma}[theorem]{Lemma}
\newtheorem{definition}[theorem]{Definition}

\newtheorem{rem}[theorem]{Remark}
\newtheorem{prop}[theorem]{Proposition}

\newcommand{\Label}[1]{\label{#1}}

\usepackage{fancyhdr} 
\newcommand\RR{\mathbb{R}}

\newcommand\CC{\mathbb{C}}

\newcommand{\cE}{\mathcal{E}}
\newcommand{\cF}{\mathcal{F}}

\newcommand{\cD}{\mathcal{D}}
\newcommand{\ZZ}{\mathbb{Z}}

\newcommand{\mm}{\mathfrak{m}}
\newcommand{\dd}{\partial}
\newcommand{\al}{\alpha}

\newcommand{\si}{\sigma}
\newcommand{\be}{\beta}

\newcommand{\ga}{\gamma}
\newcommand{\OO}{\Omega}
\newcommand{\pr}{\prime}
\newcommand{\la}{\lambda}

\newcommand{\fA}{\mathfrak A}

\newcommand{\Qp}{\mathbb Q_p}
\newcommand{\Qpn}{\mathbb Q_p^n}
\newcommand{\vph}{\varphi}
\newcommand{\Aw}{A_w(\xi)}
\newcommand{\Fx}{\cF_{x\to \xi}}
\newcommand{\Fi}{\cF^{-1}_{\xi\to x}}
\newcommand{\Fiy}{\cF^{-1}_{\xi\to y}}
\newcommand{\w}[1]{\widetilde{#1}}

\newcommand{\are}{\varkappa}

\newcommand{\col}{\colon}
\newcommand{\vep}{\varepsilon}
\newcommand{\EEE}{\mathbf E}
\newcommand{\bfi}[1]{\textbf{\textit{#1}}}

\title{Multidimensional nonlinear pseudo-differential evolution equation with $p$-adic spatial variables}
\author{\textbf{Alexandra V. Antoniouk}\\
\footnotesize Institute of Mathematics,\\
\footnotesize National Academy of Sciences of Ukraine,\\
\footnotesize Tereshchenkivska 3, Kiev, 01004 Ukraine,\\
\footnotesize E-mail: antoniouk.a@gmail.com
\and
\textbf{Andrei Yu. Khrennikov}\\
\footnotesize International Center for Mathematical Modeling\\
\footnotesize in Physics and Cognitive Sciences,\\
\footnotesize Mathematical Institute, Linnaeus University,\\
\footnotesize V\"axj\"o, SE-35195 Sweden,\\
\footnotesize E-mail: andrei.khrennikov@lnu.se
\and
\textbf{Anatoly N. Kochubei}\\
\footnotesize Institute of Mathematics,\\
\footnotesize National Academy of Sciences of Ukraine,\\
\footnotesize Tereshchenkivska 3, Kiev, 01004 Ukraine,\\
\footnotesize E-mail: kochubei@imath.kiev.ua }
\numberwithin{equation}{section}

\begin{document}
\maketitle
\bigskip
\begin{abstract}
We study the Cauchy problem for $p$-adic non-linear evolutionary pseudo-differen\-tial
equations for complex-valued functions of a real positive time variable and $p$-adic
spatial variables. Among the equations under consideration there is the $p$-adic analog of the porous medium equation (or more generally, the nonlinear filtration equation) which arise in numerous application in mathematical physics and mathematical biology. 
Our approach is based on the construction of a linear Markov semigroup on a $p$-adic ball and the proof of m-accretivity of the appropriate nonlinear operator. The latter result is equivalent to the existence and uniqueness of a mild solution of the Cauchy problem of a nonlinear equation of the porous medium type.
\end{abstract}

\vspace{2cm}


\newpage

\section{Introduction and Preliminaries}

Over a period of several hundred years theoretical physics has been developed on the basis of real and complex analysis.
However, for the past half a century the field of p-adic numbers $\Qp$ (as well as its algebraic extensions) has been intensively used in theoretical and mathematical physics (see \cite{VVZ,Khrennikov:1994,Khrennikov:1997,Ko:2001,Koz:2008} and
the references therein).
Since in $p$-adic analysis associated with the mappings from $\Qp$ to $\CC$, the operation of differentiation is not defined, many $p$-adic models instead of differential equations use pseudo-differential equations generated by so called Vladimirov operator $D^\al$.

A wide class of $p$-adic pseudo-differential equations was intensively used in applications, for example to model
basin-to-basin kinetics \cite{Avetisov:1999,Avetisov:2002,Koz:2004,Koz:2008}, for instance to construct the simplest $p$-adic pseudo-differential heat type equation \cite{Avetisov:2002}. In \cite{K:2008A} and in \cite{FZ:2006} ultrametric and $p$-adic nonlinear equations were used to model turbulence.

Some types of p-adic pseudo-differential equations were  stu\-di\-ed in detail in the books \cite{Ko:2001,AKK:book} and \cite{ZG:2016}; see also the recent survey \cite{BGPW}.

At the same time very little is known about nonlinear p-adic equations. We can mention only some semilinear evolution equations solved using p-adic wavelets \cite{AKK:book} and a kind of equations of reaction-diffusion type studied in \cite{ZG:2018}. A nonlinear evolution equation for complex-valued functions
of a real positive time variable and a $p$-adic spatial variable, which is a non-Archimedean counterpart of the fractional porous medium equation,  that is the equation:
\begin{equation}\Label{1-1}\
\dfrac{\dd u}{\dd t}+D^\al \big(\vph(u)\big)=0,\ \ u=u(t,x),\ \ t>0,\ x\in\Qp
\end{equation}
was studied in the recent paper \cite{Ko:2018}. Here $\Qp$ is the field of $p$-adic numbers, $D^\al$, $\al >0$ is Vladimirov's fractional differentiation operator, $\vph$ is a strictly monotone increasing continuous function.
Developing an $L^1$-theory of Vladimirov's $p$-adic fractional differentiation operator, the authors of \cite{Ko:2018} proved the
$m$-accretivity of the corresponding nonlinear operator and obtained the existence and uniqueness of a mild solution.

In this paper we prove a similar result for a more complicated multi-dimensional nonlinear equation resembling \eqref{1-1} but with more general nonlocal operator $W$ \eqref{W1} instead of the Vladimirov operator $D^\al$. We follow essentially the strategy of the paper \cite{Ko:2018} and use the abstract theory for nonlinear $m$-accretive operators in Banach space $L^1$ developed in \cite{BrSt:1973} and further in \cite{CrP}.

In order to use this method we need to build the $L^1$-theory for the weighted Vladimirov operator $W$ and prove some special properties of its restriction onto $p$-adic ball. This meets several difficulties, which we overcome by reconstruction of the associated Markov process in the $p$-adic ball and by recovering of the associated Levy measure. We prove several additional properties of $W$, which are necessary for our tasks and may have independent interest.

Let us formulate the main definitions and auxiliary statement, which will be used further.

\medskip

{\it 2.1. $p$-Adic numbers.} Let $p$ be a prime
number. The field of $p$-adic numbers is the completion $\mathbb Q_p$ of the field $\mathbb Q$
of rational numbers, with respect to the absolute value $|x|_p$
defined by setting $|0|_p=0$,
$$
|x|_p=p^{-\nu }\ \mbox{if }x=p^\nu \frac{m}n,
$$
where $\nu ,m,n\in \mathbb Z$, and $m,n$ are prime to $p$. $\Qp$ is a locally compact topological field.

Note that by Ostrowski's theorem there are no absolute values on $\mathbb Q$, which are not equivalent to the ``Euclidean'' one,
or one of $|\cdot |_p$.

The absolute value $|x|_p$, $x\in \mathbb Q_p$, has the following properties:
\begin{gather*}
|x|_p=0\ \mbox{if and only if }x=0;\\
|xy|_p=|x|_p\cdot |y|_p;\\
|x+y|_p\le \max (|x|_p,|y|_p).
\end{gather*}

The latter property called the ultra-metric inequality (or the non-Archi\-me\-dean property) implies the total disconnectedness of $\Qp$ in the topology
determined by the metric $|x-y|_p$, as well as many unusual geometric properties. Note also the following consequence of the ultra-metric inequality:
\begin{equation*}
|x+y|_p=\max (|x|_p,|y|_p)\quad \mbox{if }|x|_p\ne |y|_p.
\end{equation*}

The absolute value $|x|_p$ takes the discrete set of non-zero
values $p^N$, $N\in \mathbb Z$.
If $|x|_p=p^N$, then $x$ admits a
(unique) canonical representation
\begin{equation}
\Label{2.1}\
x=p^{-N}\left( x_0+x_1p+x_2p^2+\cdots \right) ,
\end{equation}
where $x_0,x_1,x_2,\ldots \in \{ 0,1,\ldots ,p-1\}$, $x_0\ne 0$.
The series converges in the topology of $\mathbb Q_p$. For
example,
$$
-1=(p-1)+(p-1)p+(p-1)p^2+\cdots ,\quad |-1|_p=1.
$$

The {\it fractional part} of element $x\in\Qp$ in canonical representation \eqref{2.1} is given by:
\[
\{x\}_p=
\left\{\begin{array}{ll} 0,&\text{if}\quad N\leq 0 \ \ \ \text{or}\ \ \ x=0;\\
p^{-N}\big(x_0+x_1p+\ldots+x_{N-1}p^{N-1}\big),&\text{if} \quad N>0.
\end{array}
\right.
\]

\medskip
{\it 2.2. Integration in the space $\Qp^n$}. The space $\Qpn=\Qp\times\cdots\times\Qp$ consists of points $x=(x_1,\ldots, x_n)$, where $x_j\in\Qp$, $j=1,\ldots,n$, $n\geq 2$. The $p$-adic norm on $\Qpn$ is
$$
\Vert x\Vert_p=\max\limits_{j=1,\ldots,n}|x_j|_p,\ \ \ \ x\in\Qpn.
$$
This norm is also non-Archimedean, since for any $x, y\in \Qpn$:
\begin{align*}
\Vert x+y\Vert_p=&\max\limits_{j=1,\ldots,n}|x_j+y_j|_p\leq \max\limits_{j=1,\ldots,n}\max\big(
|x_j|_p+|y_j|_p\big)=\\
&=\max\big(\max\limits_{j=1,\ldots,n}|x_j|_p,\max\limits_{j=1,\ldots,n}|y_j|_p\big)=\max\big(\Vert x\Vert_p,\Vert y\Vert_p\big).
\end{align*}
The space $\Qpn$ is a complete metric locally compact and totally disconnected space. The scalar product of vectors $x, y\in\Qpn$ is defined by $
x\cdot y=\sum\limits_{j=1}^n x_jy_j.
$
In the space $\Qpn$ the following change of variables formula is valid (see \cite[Ch.I, \S 4, Sect. 4, p. 68]{VVZ}). If $F: x=x(y)$ (i.e. $x_i =x_i(y_1,\ldots,y_n), i=1,\ldots,n$) is a homeomorphic map of the open compact set $K$ onto the (open) compact set $F(K)$, moreover the functions $x_i(y), i=1,\ldots,n$ are analytic in $K$ and
$
\det F^\prime (y)=\det \Big[{\dd x_j}/{\dd y_j}\Big] (y)\neq 0,$ $y\in K$,
then for any $f\in L^1(T(K))$:
\begin{equation}\Label{change}\
\int_{F(K)} f(x)\, d^nx =\int_K  f(F(y))\,\Vert\det F^\prime \Vert_p\, d^ny.
\end{equation}

{\it 2.3. Fourier transformation and generalized functions on $\Qp^n$}. The function $\chi(x)=exp\big(2\pi i\{x\}_p\big)$ is an additive character of the field $\Qp$, i.e. it is a character of its additive group, the continuous complex valued function on $\Qp$ satisfying the conditions:
\begin{equation}\Label{char}\
\vert \chi (x)\vert = 1, \quad\quad \chi(x+y)=\chi(x)+\chi(y).
\end{equation}
 Moreover $\chi(x)=1$ if and only if $| x|_p\leq 1.$ Denote by $dx$ the Haar measure on the additive group $\Qp$. Then the integral on the $p$-adic ball
 $B_N=\big\{x\in\Qpn: \Vert x\Vert_p\leq p^N\big\},\quad N\in\ZZ,$
equals (\cite[(7.14), p. 25]{VTab}):
\begin{equation}\Label{formula11}\
\int_{B_N}\chi(\xi\cdot x)\,d^n x=\left\{
\begin{array}{lcl}
p^{nN},&\text{if} &\Vert \xi\Vert_p\leq p^{-N};\\
0,&\text{if} &\Vert \xi\Vert_p> p^{-N},
\end{array}
\right.\end{equation}
Lemma 4.1 in \cite[Ch.III, \S 4, p. 137]{Taib} implies that for $z_0$ such that $\Vert z_0\Vert_p=1$:
\begin{equation}\Label{chi-1}\
\int_{S_j}\chi(z_0\cdot x)\,d^nx=\left\{
\begin{array}{rcl}
p^{jn}(1-p^{-n}),&\ &{j\leq}\  0;\\
-p^{\,{n}}p^{-n},&\ &{j=1};\\
0,&\ &{j>1},
\end{array}
\right.
\end{equation}
where $S_j=\big\{x\in\Qpn: \Vert x\Vert_p= p^{\,j}\big\},\quad j\in\ZZ$.

The {\it Fourier transform} of a test function $\vph\in\cD (\Qp^n)$  is defined by the formula
\[
\widehat{\vph}(\xi)\equiv\big(\Fx\,\vph\big)(\xi)=\int_{\Qp^n}\chi(\xi\cdot x)\vph(x)d^nx,\quad \xi \in \Qp^n,
\]
where $\chi (\xi\cdot x)= \chi(\xi_1 x_1)\cdots\chi(\xi_n x_n)=\exp(2\pi i\sum_{j=1}^n\{\xi_j x_j\}_p)$.

Remark that the additive group of $\Qp$ is self-dual, so that the Fourier transform of a complex-valued function $\vph\in \Qp^n$ is again a function on $\Qpn$ and if $\Fx\vph\in L^1(\Qpn)$ then it is true the inversion formula $\vph(x)=\int_{\Qpn}\chi(-x\cdot \xi)\widehat{\vph}(\xi)\,d^n\xi.$
Let us also remark that, in contrast to the Archimedean situation,  the Fourier transform
$\vph \to \Fx \vph$
is a linear and continuous automorphism of the space $\cD (\Qpn)$ (cf. \cite[Lemma 4.8.2]{AKK:book}, see also \cite[Ch. II,§2.4.]{Gel}, \cite[III,(3.2)]{Taib}, \cite[VII.2.]{VVZ}, i.e.
$
\vph(x)=\cF^{\,-1}_{\xi\to x}\Big(\cF_{x\to\xi} \vph \Big).
$
Here $\cD (\Qpn)$
denotes the vector space of {\it test functions}, i.e. of all locally constant functions with compact support.
 Recall that a function $\vph: \Qpn\to \CC$ is {\it locally constant} if there exist such an integer $\ell \geq 0$ that for any $x\in\Qpn$
\[\vph (x+y)=\vph (x),\quad \text{if}\quad \Vert y\Vert_p\leq p^{-\ell},\quad \text{($\ell$ is independent on $x$).}\]

The smallest number $\ell$ with this property is called {\it the exponent of constancy of the function $\vph$.}  Note that $\cD(\Qpn)$ is dense in $L^q(\Qpn)$ for each $q\in [1,\infty)$.

Let us also introduce the subspace $D_N^\ell\subset \cD(\Qpn)$ consisting of functions with supports in a ball $B_N$ and with the exponents of local constancy less than $\ell$. Then the topology in $\cD(\Qpn)$ is defined as the double inductive limit topology, so that
\[\cD(\Qpn)=\lim\limits_{\longrightarrow\atop{N\to\infty}}
\lim\limits_{\longrightarrow\atop{\ell\to\infty}}D_N^\ell.\]

If $V\subset \Qpn$ is an open set, the space $\cD(V)$ of test functions on $V$ is defined as a subspace of $\cD(\Qpn)$ consisting of functions with supports in $V$. For a ball $V=B_N$, we can identify $\cD (B_N)$ with the set of all locally constant functions on $B_N$.

The space $\cD^\pr(\Qpn)$ of Bruhat-Schwartz distributions on $\Qpn$ is defined as a strong conjugate space to $\cD(\Qpn)$. By duality, the Fourier transform is extended to a linear (and therefore continuous) automorphism of $\cD^\pr(\Qpn)$. For a detailed theory of convolutions and direct product of distributions on $\Qpn$ closely connected with the theory of their Fourier transforms see \cite{AKK:book,Ko:2001,VVZ}

\medskip
{\it 2.4. Radial operator on $\Qp^n$}. In this paper we consider the class of non-local operators introduced in \cite{ZG:2016}.

\begin{definition}\rm
Let us fix a function $w:\Qpn\to \RR_+,$
which satisfies the following properties:
\begin{itemize}
\item[(i)] $w$ is radial, i.e. depending on $\Vert y\Vert_p$, $w=w\big(\Vert y\Vert_p\big)$, continuous and increasing function of $\Vert y\Vert_p$;
\item[(ii)] $w(0)=0$, if $y=0$;
\item[(iii)] there exist such constants $C>0$  and $\al > n$ that
\begin{equation}\Label{w2}\
C_1\Vert \xi\Vert_p^\al\leq w(\Vert \xi\Vert_p)\leq C_2\Vert \xi\Vert_p^\al, \ \text{for any}\ \ \xi\in \Qpn.
\end{equation}

\end{itemize}
\end{definition}

Remark that condition (iii) implies that for some $M\in \ZZ$
\[\int\limits_{\Vert y\Vert_p\geq p^M}\dfrac{d^ny}{w\big(\Vert y\Vert_p\big)}<\infty.\]

The nonlocal operator $W$ is defined by
\begin{equation}\Label{W1}\
(W\vph)(x)=\varkappa\int_{\Qpn}\dfrac{\vph(x-y)-\vph(x)}{w(\Vert y\Vert_p)}\,d^ny,\ \ \text{for}\ \ \vph\in\cD(\Qpn),
\end{equation}
where $\are$ is some positive constant.
From \cite[(2.5), p.15]{ZG:2016} it follows that
for $\vph\in\cD(\Qpn)$ and some constant $M=M(\vph)$ operator $W$ has the following representation:
\begin{equation}\Label{W2}\
(W\vph)(x)=\varkappa\,\dfrac{1_{\Qpn\backslash B_M}}{w(\Vert x\Vert_p)}\ast  \vph(x)\, -\, \vph(x) \int\limits_{\Vert y\Vert_p\, >p^M}\dfrac{d^n y}{w(\Vert y\Vert_p)}.
\end{equation}

Moreover from Lemma 4 and Proposition 7, Ch.2 in \cite{ZG:2016} it follows that this operator, acting from $\cD(\Qpn)$ to $L^q(\Qpn)$, is linear bounded operator for each $q\in [1,\infty)$ and has the representation:
\begin{equation}\Label{FW}\
(W\vph)(x)=-\are \Fi\,\big(\Aw\,\Fx\,\vph\big),\ \text{for}\ \vph\in\cD(\Qpn),
\end{equation}
where
\begin{equation}\Label{Aw}\
\Aw:= \int_{\Qpn}\dfrac{1-\chi(y\cdot\xi)}{w(\Vert y\Vert_p)}\, d^ny.
\end{equation}

From \cite[(2.9), p.16]{ZG:2016} for any $z\in\Qpn$ and $\Vert z\Vert_p=p^{-\ga}$ it follows that
\begin{align}\Label{Awfin}\
A_w(z)=&(1-p^{-n})\sum\limits_{j=2}^{\infty}\dfrac{p^{\ga n+jn}}{w(p^{\ga+j})}\,
+ \dfrac{p^{\ga n+n}}{w(p^{\ga+1})}=\\
\Label{eq:Awrep-1}\
&=(1-p^{-n})\sum\limits_{j=\ga+2}^{\infty}\dfrac{p^{nj}}{w(p^{\,j})}\,
+ \dfrac{p^{n(\ga +1)}}{w(p^{\ga +1})}.
\end{align}

The condition \eqref{w2}
on function $w$ implies that there exist positive constants $C_3$ and $C_4$ such that
\begin{equation}\Label{Aw1}\
C_3\Vert\xi\Vert_p^{\al-n}\leq \Aw \leq C_4\Vert \xi\Vert_p^{\al-n},
\end{equation}
(see Lemma 8, Ch.2 in \cite{ZG:2016}).

Remark also that $W\vph\in C(\Qpn)\cap L^q(\Qpn)$ for each $q\in [1,\infty)$ and $\vph\in \cD(\Qpn)$ and operator $W$ may be extended to a densely defined operator in $L^2(\Qpn)$ with the domain
\[Dom (W)=\big\{\vph\in L^2(\Qpn): \ \Aw\Fx\,\vph\in L^2(\Qpn)\big\}.\]
The operator $\big(-W, Dom (W)\big)$ is essentially self-adjoint and positive. In particular, it generates a $C_0$-semigroup of contractions $T(t)$ in the space $L^2(\Qpn)$ (Proposition 20 in \cite{ZG:2016}):
\begin{equation}\Label{Tt}\
T(t)u=Z_t\ast u=\int_{\Qpn}Z(t,x-y)u(y)\,d^ny,\ \ t >0; \ T(0)u=u.
\end{equation}
Here $Z_t(x)=Z(t,x)$
\begin{equation}\Label{Zt}\
Z(t,x)=\int_{\Qpn}e^{-\are t\Aw}\chi(-x\cdot \xi)\,d^n\xi,\ \ \text{for}\ \ t>0.
\end{equation}
is the heat kernel or fundamental solution of the corresponding Cauchy problem.
Later we need the following properties of the fundamental solution $Z(t,x)$  (Lemma 10 and 11 and Theorem 13, Ch. 2 in \cite{ZG:2016}):
\begin{align}\Label{Z0}\
1) &Z(t,x)\geq 0; \quad Z_t(x)\in L^1(\Qpn),\  \text{for}\  t>0\\
\Label{Zt5}\
2) &\int_{\Qpn}Z(t,x)\,d^nx =1;\\
\Label{Zt4}\
3) & Z(t+s,x)=\int_{\Qpn} Z(t,x-y)\,Z(s,y)\, d^ny,\ \ t, s >0, \ x\in\Qpn;\\
 \Label{Zt1}\
4)&
Z(t,x)=\Fi\,\big[e^{-\are t\Aw}\big]\in C(\Qpn;\RR)\cap L^1(\Qpn)\cap L^2(\Qpn);\\
\Label{Zt3}\
5) &  Z(t,x) \leq \max\{2^\al C_1, 2^\al C_2\}\, t
\Big(\Vert x\Vert_p+
t^{\frac{1}{\al - n}}\Big)^{-\al},\ \ \text{for}\ \ t>0\ \ \text{and}\ \ x\in\Qpn;\\
\Label{Zt6}\
6)& D_tZ(t,x)=-\varkappa \int_{\Qpn}A_w(\xi)\,e^{-\varkappa t A_w(\xi)}\chi(x\cdot \xi) \, d^n \xi,\ \ \text{for}\ \ t>0,\ x\in \Qpn;\\
\Label{Zt7}\
7)& Z(t,x) =\Vert x\Vert_p^{-n}\Bigg[(1-p^{-n})\sum\limits_{j=0}^\infty p^{-nj}e^{-\are t A_w(p^{-(\be+j)})}-e^{-\are tA_w(p^{-\be +1})}\Bigg], \ \text{if}\ \ \Vert x\Vert_p=p^\be.
\end{align}

Since $Z_t(x)\in L^1(\Qpn)$ for $t>0$, for $u\in\cD(\Qpn)\subset L^\infty(\Qpn)$ the convolution in \eqref{Tt} exists and is a well-defined continuous in $x$ function (See Th, 1.11).

\medskip
{\it 2.5. Associated Markov processes}.\  \Label{Sec2-4}\  We use the analytic definition of a Markov process, that is a definition of a transition probability. (See for example \cite{Dyn}.) Suppose that $(E,\mathcal E)$ is a measurable space. A family of real-valued non-negative measurable with respect to variable $x$ functions $P(s,x;t,\Gamma)$, $s<t$, $x\in E,\Gamma \in \mathcal E$, such that $P(s,x;t,\cdot )$ is a measure on $\mathcal E$ and $P(s,x;t,\Gamma )\le 1$, is called a {\it transition probability}, if the Kolmogorov-Chapman equality
$$
\int\limits_EP(t,y;\tau ,\Gamma )\, P(s,x;t,dy)=P(s,x;\tau,\Gamma)
$$
holds whenever $s<t<\tau$, $x\in E$, $\Gamma \in \mathcal E$.

We consider $\cE=\Big(\Qpn,\Vert \cdot \Vert_p\Big)$ as a complete non-Archimedean metric space. Let $\cE$ denote the Borel $\si$-algebra   In Chapter 2.2.5 (Lemma 14, p. 22) of \cite{ZG:2016} is was shown that the transition probability
\[P(t,x,B)=\left\{\begin{array}{ll}
\int\limits_{B} p\,(t,x,y)\,d^n y,&\ \text{for} \ t >0;\ x,y\in \Qpn,\ B\in\cE\\
1_B(x),& \ \text{for} \ t=0,
\end{array}\right.
\]
with $p(t,x,y):=Z(t,x-y)$ and $Z(t,x)$ defined in \eqref{Zt},
is normal, i.e. $
\lim\limits_{t\downarrow s}P(s,x;t,\Qpn)=1
$, for any $s>0$, $x\in \Qpn$.

Due to Theorem 16 in \cite[\S 2.2.5, p.23]{ZG:2016} as the consequence of Theorem 3.6 in \cite[p. 135]{Dyn} $Z(t,x)$ is the transition density of a time and space homogeneous Markov process $\xi_t$ which is bounded, right-continuous and has no discontinuities other than jumps.
 Moreover the associated semigroup \eqref{Tt}
is Feller one.
Moreover, due to \cite[Vol. II, Ch1, \S 1, p.46] {GS} this process has an independent increments as  a  space and time homogeneous Markov process. (See also \cite[Vol. I, Ch. III, \S 1, p. 188]{GS})

\section{Properties of the operator W in Banach space $L^1(\Qpn)$}

Let us consider operator $T(t)$ defined by \eqref{Tt}:
\[\big(T(t)u\big)(x)=\int_{\Qpn} Z(t,x-\xi)\,u (\xi)\, d\xi\]
in Banach space $L^1(\Qpn)$. From \eqref{Zt4}, \eqref{Zt3} and Young inequality it follows that $T(t)$ is a contraction semigroup in $L^1(\Qpn)$.

\begin{lemma} \Label{l3-1}\ $T(t)$ is a strongly continuous semigroup in $L^1(\Qpn)$.
\end{lemma}

\begin{proof} Since the space $\cD(\Qpn)$ of Bruhat-Schwartz functions is dense in $L^1(\Qpn)$ \cite{VVZ} it is sufficient to prove the convergence of the expression $I_t:=\Vert T(t)u -u\Vert_{L^1(\Qpn)}\to 0,\ \ t\to 0$  only for $u\in\cD(\Qpn)$.

Due to \eqref{Zt5} we have for $u\in\cD(\Qpn)$:
\begin{align*}
&I_t:=\int_{\Qpn}\big\vert T(t)u(x)-u(x)\big\vert\,d^nx = \int_{\Qpn}
\Big\vert \int_{\Qpn} Z(t, x-\xi)u(\xi)\,d^n\xi -u(x)\Big\vert\,d^n x=\\
&=\int_{\Qpn}
\int_{\Qpn} Z(t,x-\xi)\,\big\vert u(\xi) -u(x)\big\vert\,d^n\xi\,d^n x.
\end{align*}
Here we also used \eqref{Z0}. Since function $u\in\cD(\Qpn)$ is locally constant, then such is also $u(\xi)-u(x)$ and therefore
it exists some $m>0$ such that $u(\xi)-u(x)=0$ for $\Vert x-\xi\Vert_p\leq p^m$. Moreover there exists some $N>m$ such that $u(x)=0$ for $\Vert x\Vert_p>p^N$.
Noting this let us represent integral $I_t$ as a sum of two integrals: inside and outside the $p$-adic ball.
Since on the ball $\Vert x-\xi\Vert_p\leq p^m$ functions $u(x)$ and $u(\xi)$ coincides, therefore, using \eqref{Zt3}, we have
\begin{align*}\nonumber
&I_{1,t}=\int\limits_{\Vert x\Vert_p\leq p^N}d^nx\int\limits_{\Vert x-\xi\Vert_p>p^m} Z(t,x-\xi)\,\vert u(\xi)-u(x)\vert\,d^n\xi{\leq}\\
\nonumber
&\leq C\,t \int\limits_{\Vert x\Vert_p\leq p^N}d^nx\int\limits_{\Vert x-\xi\Vert_p>p^m}
\Big(t^{\frac{1}{\al-n}}+\Vert x-\xi\Vert_p\Big)^{-\al}d^n\xi\leq\\
&\leq  C\,t \int\limits_{\Vert x\Vert_p\leq p^N}d^nx\int\limits_{\Vert z\Vert_p>p^m}\Vert z\Vert^{-\al}_pd^nz\to 0,\ \ t\to 0, \ \ \text{for}\ \  \al > n.
\end{align*}
Since outside the ball, for $\Vert x\Vert_p >p^N$, function $u$ vanishes, $u (x)=0$, we have
\begin{align*}
&I_{2,t}=\int\limits_{\Vert x\Vert_p> p^N}d^nx\int\limits_{\Qpn} Z(t,x-\xi)\,\vert u(\xi)-u(x)\vert\,d^n\xi=\int\limits_{\Vert x\Vert_p> p^N}d^nx\int\limits_{\Qpn} Z(t,x-\xi)\,\vert u(\xi)\vert\,d^n\xi=\\
&=\int\limits_{\Vert x\Vert_p> p^N}d^nx\int\limits_{\Vert \xi\Vert_p\leq p^N} Z(t,x-\xi)\,\vert u(\xi)\vert\,d^n\xi.
\end{align*}
On the last step we also used that support of function $u(\xi)$ is contained in the ball $\Vert \xi\Vert_p\leq p^N$. Finally, using \eqref{Zt3} we continue: 
\begin{align*}
&I_{2,t}\leq C\,t \int\limits_{\Vert x\Vert_p> p^N}d^nx\int\limits_{\Vert \xi\Vert_p\leq p^N}
\Big(t^{\frac{1}{\al-n}}+\Vert x- \xi\Vert_p\Big)^{-\al}d^n\xi=\\
&=C\,t \int\limits_{\Vert x\Vert_p> p^N}d^nx\int\limits_{\Vert \xi\Vert_p\leq p^N}
\Big(t^{\frac{1}{\al-n}}+\Vert x\Vert_p\Big)^{-\al}d^n\xi\to 0, \ \ \text{as}\ \  t\to 0,
\end{align*}
where we have used that $\Vert x-\xi\Vert_p=\Vert x\Vert_p$ for $\Vert x\Vert_p>p^N$.
\end{proof}

\begin{definition}\Label{def:3-2}\
Let us define operator $\mathfrak A$ as a generator of semigroup $T(t)$ in space $L^1(\Qpn)$ and let $Dom\,(\fA)$ be its domain.
\end{definition}

\begin{lemma}\Label{lem3-2} Any test function $u\in \cD(\Qpn)$ belongs to the domain of the operator $\fA$ in $L^1(\Qpn)\colon \cD(\Qpn) \subset Dom(\fA)$. Moreover, on the test functions the operator $\fA$ coincides with the representation of operator $W$ of the form \eqref{FW}.
\end{lemma}

\begin{proof} Let us write for $u\in \cD(\Qpn)$:
\begin{align*}
&\dfrac{1}{t}\Big( T(t)u-u\Big)=\frac{1}{t}\big(Z_t\ast u-u\big)=\\
&=\dfrac{1}{t}\Big(\Fi\big[e^{-\are t \Aw}\big]\ast u - \Fi \Fx u\big] \Big)=\\
&=\dfrac{1}{t}\Big(\Fi \big[e^{-\are t \Aw}\big]\ast \Fi \Fx u -
\Fi \Fx u\big] \Big)=\\
&=\dfrac{1}{t}\Big(\Fi \big[e^{-\are t \Aw}\cdot \Fx \,u\big] -
\Fi \Fx u\big] \Big)=\\
&=\Fi \Big[ \dfrac{1}{t}\big(e^{-\are t \Aw}-1\big) \Fx \,u\Big].
\end{align*}
Taking into account these calculations, to finish the prove we need to show that
\[F_t:=\int_{\Qpn} \Big\vert \Fi \big[ \dfrac{1}{t}\big(e^{-\are t \Aw}-1\big)\Fx\, u +\are \Aw\Fx\,u\big]\Big\vert\,d^n\xi \to 0, t\to 0.\]
To see this it is sufficient to note that for $u\in\cD(\Qpn)$ such that the support of $\Fx \,u$ is contained in some ball $B_N$:
\begin{align*}
&\big\vert\dfrac{1}{t}\big(e^{-\are \,t\Aw}-1\big)+\are\Aw \big\vert\cdot\vert\Fx\, u\vert =\\
&= \Big\vert\dfrac{1}{t}
\Big(
\sum\limits_{q=0}^\infty (-1)^q\dfrac{\big(\are\,t \Aw\big)^q}{q!} -1\Big)+\are \Aw \Big\vert\cdot \vert \Fx\,  u\vert=\\
&= t\, A_w^2(\xi) \are^2\big\vert \sum\limits_{m=2}^\infty (-1)^m\dfrac{\are^m t^m A_w^m(\xi)}{(m+2)!}\big\vert\cdot\vert\Fx\,  u\vert\leq\\
&\leq C_N \, t\, A_w^2(\xi) \vert\Fx\,  u\vert
\end{align*}
since the series converges for any $\xi$ which belongs to the support of $u$.
Therefore
\begin{align*}F_t
& \leq \int\limits_{\Qpn}\Big\vert \Big( \dfrac{1}{t}\big(e^{-\are t \Aw}-1\big) +\are\Aw \Big)\Fx \,u\big\vert\, d^n\xi\leq\\
&\leq C_N\,t\int\limits_{\Vert \xi\Vert_p\leq\, p^{-N}}A^2_w(\xi)\,\vert\Fx u\vert \, d^n\xi \to 0, \ t\to 0.
\end{align*}
\end{proof}

\section{Construction of the Markov process in the ball}

Let $\xi_t$ be the Markov process on $\Qpn$ constructed in 2.4 of Section \ref{Sec2-4}. Like in \cite{Ko:2001} (Section 4.6.1) we construct the Markov process on a ball $B_N=\{x\in\Qpn\col \Vert x \Vert \leq p^N\}.
$
Suppose that $\xi_0\in B_N$. Denote by $\xi^{(N)}_t$ the sum of all jumps of the process
$\xi_\tau$, $\tau\in [0,t]$, whose absolute values exceed $p^N$. Since $\xi_t$ is right continuous process with left limits, $\xi_t^{(N)}$ is finite a.s. Moreover
$\xi_0^{(N)}=0$. Let us consider process
\begin{equation}\Label{etat}\
\eta_t=\xi_t-\xi_t^{(N)}.
\end{equation}
Since the jumps of $\eta_t$ never exceed $p^N$ by absolute value, this process remain a.s. in $B_N$ (due to the ultra-metric inequality).

Below we will identify a function on ball $B_N$ with its extension by zero
onto the whole $\Qpn$. Let $\cD(B_N)$ consist of functions from
$\cD(\Qpn)$ supported in the ball $B_N$. Then $\cD(B_N)$ is dense in
$L^2(B_N)$, and the operator $W_N$ in $L^2(B_N)$ defined by restricting
the operator $W$ \eqref{W1} to $\cD (B_N)$ and considering
$\big(Wu\big)(x)$ only for $x\in B_N$.
Let us find the generator $W_N$ of process $\eta_t$. To do this we need the following bunch of  lemmas.
\begin{lemma}\Label{lem4-1}\ For any $z\in \Qpn$
\begin{equation}\Label{eq:4-1}\
\EEE \chi\big(z\cdot\xi_t\big)=\exp \big(-t\varkappa A_w(z)\big).
\end{equation}
\end{lemma}
\begin{proof} Indeed, let $\mu_t$ denote the distribution of the process $\xi_t$ with $\xi_0=0$. Then, due to \cite[Vol.2, p. 24]{GS}, the Fourier transform of measure $\mu_t$ is equal to
\begin{align*}
&\EEE\, \chi\big(z\cdot\xi_t\big)=\widehat{\mu}_t(z)=\int_{\Qpn} \chi(z\cdot y)\,\mu_t(dy)=\int_{\Qpn} \chi(z\cdot y)\,p(t,0,y)\,d^ny=\\
&=\int_{\Qpn} \chi(z\cdot y)\,Z(t,y)\,d^ny=\int_{\Qpn} \chi(z\cdot y)\,\Fiy\Big[e^{-\varkappa \,t A_w(\xi)}\Big] \,d^ny=e^{-\varkappa \,t A_w(z)}.
\end{align*}
Here we also used \eqref{Zt1} and that the Fourier transform $\cF$ is a linear continuous automorphism of the space $\cD (\Qpn)$.
\end{proof}
\begin{lemma}\Label{Iz}\
\begin{equation}\Label{eq:Iz}\
I_{B_N}(z)=\int_{B_N}\dfrac{1-\chi(z\cdot x)}{w(\Vert x\Vert_p)} \,d^nx=\left\{
\begin{array}{lcl}
0&,\text{if} &\Vert z\Vert_p \leq p^{-N};\\
A_w(z)-\la_N&,\text{if} &\Vert z\Vert_p > p^{-N},
\end{array}
\right.
\end{equation}
where
\begin{equation}\Label{la}\
\la_N  =(1-p^{-n})\sum\limits_{j =\, N+1}^\infty
 \dfrac{p^{\,nj}}{w(\,p^{\,j})}.
\end{equation}
\end{lemma}
\begin{rem} \rm Remark that
\begin{equation}\Label{la-1}\
\la_N= \int\limits_{\Vert y\Vert_p >p^N}\dfrac{d^ny}{w(\Vert y\Vert_p)}.
\end{equation}
{\it Indeed.}
\begin{align*}
\la_N=\int\limits_{\Vert y\Vert_p \,>\,p^N}\dfrac{d^ny}{w(\Vert y\Vert_p)}= \sum\limits_{j=N+1}^\infty\,\int\limits_{S_j} \dfrac{d^ny}{w(\,p^{\,j})}=(1-p^{-n})\sum\limits_{j =\, N+1}^\infty
 \dfrac{p^{\,nj}}{w(\,p^{\,j})}.
\end{align*}
\end{rem}
\begin{rem}\rm
$$
I_{\Qpn} (z) = A_w(z).
$$
\end{rem}
\begin{proof} First of all remark that the character $\chi$ equals $1$ on ball $B_0$. Hence $\EEE \,\chi(z\cdot\eta_t)=1$ if $\Vert z\Vert_p\leq p^{-N}.$

Let us consider $z\in\Qpn$ such that $\Vert z\Vert_p =p^k$, $k\geq -N+1$.
Any such $z$ belongs to $S_k$, $k\geq -N+1$ and we may be represent it as $z=p^{-k} z_0$, where $\Vert z_0\Vert_p=1$. Since
$$
B_N\backslash \{0\} = \bigsqcup_{j\leq N} p^j S_0=\bigsqcup_{j\leq N} S_j,
$$
we may represent in the following form:
\begin{align*}
I_{B_N}(z)&=
\int\limits_{B_N}\dfrac{1-\chi(z\cdot x)}{w(\Vert x\Vert_p)} \,d^nx
=\sum\limits_{j=-\infty}^N\int\limits_{S_j}\dfrac{1-\chi(z\cdot x)}{w(\Vert x\Vert_p)} \,d^nx=\sum\limits_{j=-\infty}^N\int\limits_{S_j}\dfrac{1-\chi(p^{-k} z_0\cdot x)}{w(\Vert x\Vert_p)} \,d^nx.
\end{align*}
The change of variables formula \eqref{change} implies
\begin{align*}
I_{B_N}(z)=&\sum\limits_{j=-\infty}^Np^{-k n}\int\limits_{S_{j+k}}\dfrac{1-\chi( z_0\cdot y)}{w(\,p^{-k}\Vert y\Vert_p)} \,d^ny=\\
=& \sum\limits_{j=-\infty}^N\dfrac{p^{-k n}}{w(\,p^{\,j})}
\Big\{ p^{(j+k)n}(1-p^{-n})-\int\limits_{S_{j+k}}\chi( z_0\cdot y)\,d^ny\Big\}=\\
=&\sum\limits_{\ell=-\infty}^{k+N}\dfrac{p^{-k n}}{w(\,p^{\ell-k})}
\Big\{ p^{\ell n}(1-p^{-n})-\int\limits_{S_\ell}\chi( z_0\cdot y)\,d^ny\Big\}.
\end{align*}
Using \eqref{chi-1} we have for $z$ such that $\Vert z\Vert_p =p^k$, $k\geq -N+1$
\begin{align*}
I_{B_N}(z)=&\left\{
\begin{array}{lcl} \sum\limits_{\ell=-\infty}^{0}\dfrac{p^{ -kn}p^{\ell n}}{w(\,p^{\,\ell - k})}
\Big\{(1- p^{-n})-(1-p^{-n})\Big\}=0,&\ &\ell\ \leq\  0;\\
\dfrac{p^{-kn}p^{\,{n}}}{w(\,p^{\, {1}-k})}
\Big\{(1- p^{-n})- (p^{-n})\Big\},&\ &\ell ={1};\\
\sum\limits_{\ell =2}^{{k+N}}\dfrac{p^{-kn}p^{\,\ell n}}{w(\,p^{\,\ell - k})}
(1-p^{-n}),&\ &\ell\ > 1,
\end{array}
\right.
\end{align*}
and finally 
\[
I_{B_N}(z)=\sum\limits_{j=2}^{{k+N}}\dfrac{p^{-k n+jn}}{w(\,p^{\,j-k})}
(1-p^{-n})+\dfrac{p^{-k n{+}n}}{w(\,p^{\,{1}-k})}.
\]
Comparing this with \eqref{Awfin} we conclude that
\begin{align*}
I_{B_N}(z)=& A_w(z) - (1-p^{-n})\sum\limits_{j= k+N+1}^\infty \dfrac{p^{-k n+jn}}{w(\,p^{\,j-k})}=\\
=& A_w(z) - (1-p^{-n})\sum\limits_{\ell = N+1}^\infty
 \dfrac{p^{\,\ell \,n}}{w(\,p^{\,\ell})}=  A_w(z) -\la_N,
\end{align*}
where $\la_N$ is given by \eqref{la}.
\end{proof}
\begin{lemma}\Label{lem:4-6}\ For any $\psi\in \cD (\Qpn)$ such that its Fourier transform $u = \cF \psi$ has a support in the ball $B_N$ we have:
\begin{equation}\Label{eq:4-14}\
\int\limits_{\Vert z\Vert_p\,\leq\, p^{-N}}A_w(z)\,\psi(z)\,d^nz=\la_N\int\limits_{\Vert z\Vert \,\leq \, p^{-N}} \psi(z)\,d^nz
\end{equation}
\end{lemma}
\begin{proof} Since $\text{\rm supp} \,\psi\subset B_{-N}$, we have $\psi(z)=\psi(0)$ for $\Vert z\Vert_p\leq p^{-N}$.  Therefore
\begin{align}\nonumber
J_N&:=\int\limits_{\Vert z\Vert_p\,\leq\, p^{-N}}A_w(z)\psi(z)\,d^nz=
\psi(0)\int\limits_{\Vert z\Vert_p\,\leq\, p^{-N}}A_w(z)\,d^nz=\\
\nonumber
 &=\psi(0)\int\limits_{\Vert z\Vert_p\,\leq\, p^{-N}}
\int\limits_{\Qpn}\dfrac{1-\chi(y\cdot z)}{w(\Vert y\Vert_p)}\, d^ny\,d^nz=\\
\Label{4-12}\
&=\psi(0)\int\limits_{\Qpn}\int\limits_{\Vert z\Vert_p\,\leq\, p^{-N}}
\dfrac{1-\chi(y\cdot z)}{w(\Vert y\Vert_p)}\, d^nz\,d^ny.
\end{align}
Since the character $\chi$ equals $1$ on ball $B_0$, hence the integral above is equal zero, when $\Vert y\cdot z\Vert_p=\Vert y\Vert_p\cdot \Vert z\Vert_p\leq 1$. Therefore, for $\Vert z\Vert_p\leq p^{-N}$, if $\Vert y\Vert_p\leq p^N$, we have
$$
\Vert y\cdot z\Vert_p=\Vert y\Vert_p\cdot \Vert z\Vert_p\leq 1
$$
and therefore integral in \eqref{4-12} may be represented as:
\begin{align*}
J_N&=\psi(0)\int\limits_{\Vert y\Vert_p >p^N}\int\limits_{\Vert z\Vert_p\,\leq\, p^{-N}}
\dfrac{1-\chi(y\cdot z)}{w(\Vert y\Vert_p)}\, d^nz\,d^ny=\\
&=\psi(0)\int\limits_{\Vert y\Vert_p >p^N} \dfrac{d^ny}{w(\Vert y\Vert_p)}
\sum\limits_{j=-\infty}^{-N}\,\int\limits_{S_j}\Big(1-\chi(y\cdot z)\Big)\,d^nz=\\
&=\psi(0)\sum\limits_{k=N+1}^\infty\,\int\limits_{S_k} \dfrac{d^ny}{w(\Vert y\Vert_p)}
\sum\limits_{j=-\infty}^{-N}\,\int\limits_{S_j}\Big(1-\chi(y\cdot z)\Big)\,d^nz.
\end{align*}
Let $\Vert y\Vert_p=p^{\,k}$, $k\geq N+1$, then $y=p^{-k}y_0$, where $\Vert y_0\Vert_p=1$, and change  of variables formula \eqref{change}
gives for the expression under integral:
\begin{align}\Label{4-15}\
\nonumber
&S_N(y)=\sum\limits_{j=-\infty}^{-N}\,\int\limits_{S_j}\Big(1-\chi(y\cdot z)\Big)\,d^nz=\sum\limits_{j=-\infty}^{-N}\,\int\limits_{S_j}\Big(1- \chi(p^{-k}y_0\cdot z)\Big)\,d^nz=\\
&=\sum\limits_{j=-\infty}^{-N}\,p^{-kn}\int\limits_{S_{j+k}}\Big(1-\chi(y_0\cdot v)\Big)\,d^nv=\sum\limits_{\ell =-\infty}^{k-N}\,p^{-kn}\int\limits_{S_{\ell}}\Big(1-\chi(y_0\cdot v)\Big)\,d^nv.
\end{align}
Due to \eqref{chi-1} using that
\begin{equation}\Label{Sell}\
\int\limits_{S_\ell}d^nz=(1-p^{-n})p^{n\ell}
\end{equation}
we have
\begin{align*}
S_N(y)&=\left\{
\begin{array}{lcl}
\sum\limits_{\ell =-\infty}^{0}\,p^{-kn}\big[(1-p^{-n})p^{n\ell}-p^{n\ell} (1-p^{-n})\big]=0,&\ &{\ell \leq}\  0;\\
p^{-kn}\big[(1-p^{-n})p^{n}+p^np^{-n}\big]=p^{-kn}p^n,&\ &{\ell =1};\\
\sum\limits_{\ell =2}^{k-N}\ p^{-kn}p^{n\ell}(1-p^{-n}),&\ &{\ell >1}.
\end{array}
\right.
\end{align*}
Finally, for $y$ such that $\Vert y\Vert_p=p^k$, $k\geq N+1$ we may continue \eqref{4-15}:
\begin{align*}
S_N(y)=&\sum\limits_{j=-\infty}^{-N}\,\int\limits_{S_j}\Big(1- \chi(y\cdot z)\Big)\,d^nz=\sum\limits_{\ell =2}^{k-N}\,p^{-kn}\,p^{\ell n}(1-p^{-n}) \,+\,p^{-kn}p^n=\\
&=
p^{-kn}\,\Big(\sum\limits_{\ell =2}^{k-N}p^{\ell n}-\sum\limits_{\ell =2}^{k-N}p^{\ell n}
p^{-n}+p^n\Big)
=p^{-kn}\,\Big(\sum\limits_{\ell =1}^{k-N}p^{\ell n}-\sum\limits_{\ell =2}^{k-N}p^{\ell n}
p^{-n}\Big)=\\
&=p^{-kn}\,\Big(\sum\limits_{\ell =1}^{k-N}p^{\ell n}-\sum\limits_{\ell =2}^{k-N}p^{n(\ell -1)}
\Big)=p^{-kn}\,\Big(\sum\limits_{\ell =1}^{k-N}p^{n\ell}-\sum\limits_{s =1}^{k-N-1}p^{ns}
\Big)=\\
&=p^{-kn}\,p^{n(k-N)}=p^{-nN}.
\end{align*}
Therefore, again using \eqref{Sell} we have
\begin{align*}
J_N&=\psi(0)\sum\limits_{k=N+1}^\infty\,\int\limits_{S_k} \dfrac{d^ny}{w(\Vert y\Vert_p)}
\sum\limits_{j=-\infty}^{-N}\,\int\limits_{S_j}\Big(1-\chi(y\cdot z)\Big)\,d^nz=\\
&=\psi(0)\,p^{-nN}\sum\limits_{k=N+1}^\infty\,\int\limits_{S_k} \dfrac{d^ny}{w(\,p^{\,k})}\,
=\psi(0)\,p^{-nN}(1-p^{-n})\sum\limits_{k=N+1}^\infty\,\dfrac{p^{nk}}{w(\,p^{\,k})}=\\
&=\psi(0)\,p^{-nN} \,\la_N=\la_N\int\limits_{\Vert z\Vert \,\leq \, p^{-N}} \psi(z)\,d^nz.
\end{align*}
where $\la_N$ is given by \eqref{la}.
\end{proof}
\begin{theorem}\Label{tm:4-5}\ If $\eta_t\vert_{t=0}=x$ and $u \in \cD (B_N)$, then
\begin{equation}\Label{gen}\
\dfrac{d}{dt}\EEE\, u(\eta_t)\big|_{t=0}=- \big(W_N u\big)(x)+\varkappa\,\la_N\,u(x),
\end{equation}
where operator $W_N$ is defined by restricting $W$ to the function $u$ supported in the ball $B_N$ and the resulting function $Wu$ is considered only on the ball $B_N$, i.e.
$$\big(W_Nu \big)(x) =\big(Wu\big)\!\!\upharpoonright_{B_N},\ \ \text{\rm for}\ \ u\in\cD(B_N).
$$
\end{theorem}
\noindent Remark that the functions on $B_N$ we identify with its extension by zero onto $\Qpn$.
\begin{rem}\rm
This theorem actually states that the operator $W_N -\varkappa\,  \la_N$ is the generator of the stochastic process $\eta_t$ located in the ball $B_N$.
\end{rem}
\begin{proof} By general result about processes with independent increments on the locally compact space \cite{Hey} the processes $\xi_t^{(N)}$ and $\eta_t$ are independent. Thus
\begin{equation}\Label{eq:4-1-1}\
\EEE \chi\big(z\cdot \xi_t\big)=\EEE\,\chi\big(z\cdot\xi_t^{(N)}\big)\cdot\EEE\chi
\big(z\cdot\eta_t\big)
\end{equation}
for any $z\in\Qpn$.
Due to Lemma \ref{lem4-1} we have
\begin{equation}\Label{4-2-2}\
\EEE \chi\big(z\cdot\xi_t\big)=\exp \big(-t\varkappa A_w(z)\big).
\end{equation}
On the other hand from Theorem 5.6.17 of \cite[p. 397]{Hey} for any
locally compact Abelian group $G=\Qpn$ having a countable basis of its
topology and the distribution $\mu_t:=P_{X_t}$ of an additive process
$\{X_t\}_{t\in[0,1]}$ on $(\OO,\cF,P)$ with values in $\Qpn$ for a chosen
fixed inner product $g$ fo $G$ there exist an element $x_t\in G$ and a
positive quadratic form $\phi_t$ on the character group $ G^{\wedge}$ of
$G$ such that
$$
\widehat{\mu}_t(\chi)=\chi(x_t)\exp\Bigg\{
-\phi_t+
\int_{G}\big[\chi( x)-1-ig(x,\chi)\big] \,\pi(t,dx) \Bigg\}
$$
holds for all $\chi\in G^{\wedge}$
for some $x_t\in \Qpn$. From Example 4 of  5.1.9 of \cite[Ch. V, p. 342]{Hey} it follows that for the totally disconnected group $G$ the zero function on $G\times G^\wedge$ is a local inner product for $G$, moreover on the totally disconnected group there is no nonzero Gaussian measure, it follows that
\begin{equation}\Label{eq:4-4}\
 \EEE\, \chi\big(z\cdot\xi_t\big)\equiv\widehat{\mu}_t(z)=
\chi({x_t}\cdot z)\exp\Big\{\int_{\Qpn}\big[\chi(z\cdot x)-1\big] \,\pi(t,dx) \Big\}.
\end{equation}
Comparing this with \eqref{4-2-2}, where due to \eqref{Aw}
we conclude that
 the Levy measure of process $\xi_t$ is equal to
\begin{equation}\Label{pi}\
\pi(t,dx)=\dfrac{\varkappa\,t }{w(\Vert x\Vert_p )}d^nx
\end{equation}
and $x_t$ in \eqref{eq:4-4} may be chosen as zero so that $\chi(x_t\cdot z)\equiv 1$.

From Proposition 1 of \cite{Ev:89} it follows that
\begin{equation}\Label{Eeta}\
\EEE \chi\big(z\cdot\eta_t\big)=\EEE \chi\big(z\cdot (\xi_t-\xi_t^{(N)})\big)=\exp\Big\{\int_{B_N}\big[\chi(z\cdot x)-1\big] \,\pi(t,dx) \Big\}.
\end{equation}
Noting \eqref{pi} and notation \eqref{eq:Iz} we may represent the integral in the exponent of \eqref{Eeta} in the following form:
\begin{align*}
\int_{B_N}\big[\chi(z\cdot x)-1\big] \,\pi(t,dx)=\varkappa\,t\int_{B_N}\dfrac{\chi(z\cdot x)-1}{w(\Vert x\Vert_p)} \,d^nx=-\varkappa\,t\, I_{B_N}(z).
\end{align*}
Due to Lemma \ref{Iz}
$$I_{B_N}(z)=\left\{
\begin{array}{lcl}
0&,\ \text{if} &\Vert z\Vert_p \leq p^{-N};\\
A_w(z)-\la_N&,\ \text{if} &\Vert z\Vert_p > p^{-N},
\end{array}
\right.
$$
with $\la_N$ given by \eqref{la} therefore from \eqref{Eeta} we have
\begin{equation}\Label{Ko:4-72}\
\EEE\,\chi \big(z\cdot \eta_t\big)=\left\{
\begin{array}{lcl}
1&,\ \text{if} &\Vert z\Vert_p \leq p^{-N};\\
\exp\big\{\varkappa t \,\big(\la_N - A_w(z)\big)\big\}&,\ \text{if} &\Vert z\Vert_p > p^{-N}.
\end{array}
\right.
\end{equation}
If we consider $u = \cF \psi$ for $\psi\in\cD(\Qpn)$, then from \eqref{Ko:4-72} we have:
\[\EEE\,u \big(\eta_t\big)=\int\limits_{\Vert z\Vert_p\leq\, p^{-N}} \psi (z)\,d^nz+
e^{\are\la_N\, t} \int\limits_{\Vert z\Vert_p\,>p^{-N}} e^{-t\varkappa A_w(z)}\psi(z)\,d^nz,\]
therefore
\begin{equation}\Label{Ko:4-73}\
\dfrac{d}{dt}\EEE\, u\big(\eta_t\big)\Big|_{t=0}=\are\la_N\int\limits_{\Vert z\Vert_p\,>p^{-N}}\psi(z)\,d^nz -\varkappa\int\limits_{\Vert z\Vert_p\,>p^{-N}}A_w(z)\psi(z)\,d^nz.
\end{equation}
By the Fourier inversion formula
$$
\int_{\Qpn}\psi(z)\,d^nz=u(0).
$$
On the other habd, since $\text{supp}\, u \subset B_N$, we find that $\psi(z)=\psi(0)$ for $\Vert z\Vert_p\leq p^{-N}$.

Lemma \ref{lem:4-6} implies that \eqref{Ko:4-73} takes the form
\begin{align*}
\dfrac{d}{dt}\EEE\, u\big(\eta_t\big)\Big|_{t=0}&=\are\la_N\int\limits_{\Qpn}\psi(z)\,d^nz -\varkappa\int\limits_{\Qpn}A_w(z)\psi(z)\,d^nz=\\
&=\are \la_N\,u(0)-\big(W_Nu\big)(0)
\end{align*}
which finished the proof of Theorem \ref{tm:4-5}.
\end{proof}
\begin{lemma}\Label{5*}\ Let the support of a function $u\in L^1(\Qpn)$ be contained in $\Qpn \backslash B_N$. Then the restriction to $B_N$ of the distribution $Wu\in \cD^\pr(\Qpn)$ coincides with the constant:
\begin{equation}\Label{Rr}\
R_N =R_N(u)=\are \int\limits_{\Qpn \backslash B_N} \dfrac{u(y)\,d^ny}{w(\Vert y\Vert_p)},
\end{equation}
i.e. for $u\in L^1(\Qpn)$, $\text{\rm supp}\, u\subset \Qpn \backslash B_N$:
$$
\big(Wu\big)\!\!\upharpoonright_{x\in B_N}= R_N(u).
$$
\end{lemma}
\begin{proof} Let $\psi \in \cD(B_N)$. Then $\langle Wu,\psi\rangle=\langle u, W\psi\rangle.$ Since $\psi (x)=0$ for $\Vert x\Vert_p>p^N$
\begin{align*}
&\langle u,W\psi\rangle =\varkappa \int\limits_{\Vert x\Vert_p>p^N} u(x)\,d^n x\int\limits_{\Qpn} \dfrac{\psi(x-y)-\psi(x)}{w(\Vert y\Vert_p)}\,d^ny\stackrel{\psi(x)=0}{=}\\
&=\varkappa \int\limits_{\Vert x\Vert_p>p^N} u(x)\,d^n x\int\limits_{\Qpn} \dfrac{\psi(x-y)}{w(\Vert y\Vert_p)}\,d^ny\stackrel{x-y=z}{=}\\
&=\varkappa \int\limits_{\Vert x\Vert_p>p^N} u(x)\,d^n x\int\limits_{\Qpn} \dfrac{\psi(z)}{w(\Vert x-z\Vert_p)}\,d^nz=\varkappa \int\limits_{\Vert x\Vert_p>p^N} u(x)\,d^n x\int\limits_{\Vert z\Vert_p\leq p^N} \dfrac{\psi(z)}{w(\Vert x-z\Vert_p)}\,d^nz=\\
&=\varkappa \int\limits_{\Vert x\Vert_p>p^N}\dfrac{ u(x)}{w(\Vert x\Vert_p)}\,d^n x\cdot \int\limits_{\Vert z\Vert_p\leq p^N} \psi(z)\,d^nz.
\end{align*}
On the last step we have used that for $\Vert x\Vert_p>p^N$ and $\Vert z\Vert_p\leq p^N$ it follows that $\Vert x-z\Vert_p=\Vert x\Vert_p$. Since $\psi\in \cD(B_N)$ is arbitrary, this implies the required property.
\end{proof}

\section{Semigroup on the $p$-adic ball}
Consider on the ball $B_N$, $N\in\ZZ$ the following Cauchy problem
\begin{align}\Label{Cauchy}\
\nonumber
&\dfrac{\dd u(t,x)}{\dd t}+\big(W_N -\varkappa  \la_N\big) u(t,x) =0,\ \ \ x\in B_N, \ \ t>0;\\
&u(0,x)=\psi (x),\ \ x\in B_N,
\end{align}
where operator is given by Theorem \ref{tm:4-5} and $\la_N$ is obtained from \eqref{la}.
Recall that operator $W_N$ is defined by restricting $W$ to the function $u_N$ supported in the ball $B_N$ and considering the resulting function $Wu_N$ only on the ball $B_N$. The functions on $B_N$ we identify with its extension by zero onto $\Qpn$:
$\big(W_N u_N \big)(x) =\big(W u_N\big)\!\!\upharpoonright_{B_N},\ \ \text{\rm for}\ \ u_N\in\cD(B_N).$

Remark that for operator $W_N$ considered on $L^2(B_N)$ $\varkappa \la_N$ is its eigenvalue. This follows from \eqref{W2} and expression for $\la_N$ given by \eqref{la-1}.

A maximum principle arguments as in the proof of Theorem 4.5 in \cite[p. 82]{Ko:2001} proved the uniqueness of the solution of Cauchy problem \eqref{Cauchy}. The fundamental solution $Z_N(t, x-y)$ for the problem \eqref{Cauchy} is the transition density of the process $\eta_t$ \eqref{etat}.
The next result gives a formula for this transition density.

\begin{theorem}\Label{ZN}\ The solution of the problem \eqref{Cauchy} is given by the formula
\begin{equation}\Label{ZN1}\
u_N(t,x) =\int_{B_N}Z_N(t,x-y)\,\psi(y)\, d^ny,\ \ t>0,\ \ x\in B_N,
\end{equation}
where
\begin{align}\Label{Zr}\
&Z_N(t,x) = e^{\varkappa\la_N t}Z(t,x) +c(t),\ \ x\in B_N,\\
\Label{ct}\
&c(t)=\dfrac{1}{\mm(B_N)}-\dfrac{e^{\varkappa\la_N t}}{\mm(B_N)}\int_{B_N}Z(t,x)\,d^nx
\end{align}
and $Z(t,x)$ is from \eqref{Zt}. Here $\mm (B_N)=\int\limits_{B_N}d^nx =p^{nN}$.
Moreover
\begin{equation}\Label{cpr}\
c^\pr(t)=-e^{\varkappa \la_N t}\varkappa\int\limits_{\Qpn\backslash B_N}\dfrac{Z(t,\xi)}{w(\Vert \xi \Vert_p)}\,d^n\xi.
\end{equation}
\end{theorem}
\begin{proof}
For any $\psi\in \cD(\Qpn)$ such that $\text{supp}\,\psi\subset B_N$ the solution to the Cauchy problem \eqref{Cauchy} for $t>0$ is given by
\begin{align*}
&u_N(t,x)=\theta_N(x)\int_{B_N}Z_N(t,x-y)\psi(y)\,d^ny=\\
& =\theta_N(x)\,e^{\varkappa \la_N t}\int_{B_N}Z(t,x-y)\psi(y)\,d^ny +\theta_N(x)\,c(t)\int_{B_N}\psi(y)\,d^ny=u_1(t,x)+u_2(t,x),
\end{align*}
where  $\theta_N(x)$, $x\in\Qpn$ is an indicator of the set $B_N$ and
\begin{align}\Label{u2}\ \nonumber
u_1(t,x)&=\theta_N(x)\,e^{\varkappa \la_N t}\int_{B_N}Z(t,x-y)\psi(y)\,d^ny;\\
u_2(t,x)&=\theta_N(x)\,c(t)\int_{B_N}\psi(y)\,d^ny.
\end{align}
Let us check that
\begin{equation}\Label{eq}\
\big(D_t+W_N -\varkappa  \la_N\big) u_N(t,x)=0
\end{equation}
for $\psi\in\cD (B_N)$ and $x\in B_N$.
We may write for $\psi\in \cD(B_N)$ and $x\in B_N$:
\begin{align}\Label{eq:5-10}\
\nonumber
&\big(D_t+W_N\big)u_N(t,x)-\varkappa \la_N\, u_N(t,x)=\\
\nonumber
&=\big(D_t+W_N\big)\Big[\theta_N(x) \,e^{\varkappa \la_N t}\int\limits_{B_N}Z(t,x-y)\psi(y)\,d^ny +\theta_N(x) \,c(t)\int\limits_{B_N}\psi(y)\,d^ny\Big]=\\
&=\big(D_t+W_N\big)\big[u_1(t,x)+u_2(t,x)\big]-\varkappa \la_N \big[u_1(t,x)+u_2(t,x)\big].
\end{align}
Let us introduce functions
\begin{align}\Label{h2*}\ \nonumber
h_1(t,x)&=\theta_N(x) \int_{B_N}Z(t,x-y)\psi(y)\,d^ny=e^{-\varkappa \la_N t}u_1(t,x);\\
h_2(t,x)&=\big( 1- \theta_N(x)\big) \int_{B_N}Z(t,x-y)\psi(y)\,d^ny
\end{align}
and remark that
$$
\big(D_t+W\big)h_1=-\big(D_t+W\big) h_2
$$
or
$$
\big(D_t+W\big)h_1=-W h_2.
$$
Since
$$
\big(D_t+W\big)h_1=\big(D_t+W\big)\big(e^{-\varkappa \la_N t}u_1\big)=e^{-\varkappa \la_N t} \big(D_t+W\big)u_1-\varkappa \la_N e^{-\varkappa \la_N t}u_1
$$
we have for $x\in B_N$ that
$$
\big(D_t+W\big)u_1-\varkappa \la_N u_1(t,x)=-e^{\varkappa \la_N t}W h_2(t,x).
$$
Therefore we may continue:
\begin{equation}\Label{eq:5-11}\
\eqref{eq:5-10}= \big(D_t+W_N\big)u_2(t,x)-\varkappa \la_N u_2(t,x)-e^{\varkappa \la_N t}W h_2(t,x),
\end{equation}
where $c(t)$ is given by \eqref{ct}.

From \eqref{W2} it follows that function $\theta_N(x)$ is an eigenfunction of the operator $W_N$ corresponding to the eigenvalue $\varkappa\,\la_N$ with $\la_N$ defined in \eqref{la-1}.
Therefore, taking into account the representation \eqref{u2} for $u_2(t,x)$, we have
\begin{equation}\Label{Wru2}\
\big(D_t+W_N\big)u_2(t,x)-\varkappa \la_N u_2(t,x)=c^\pr(t)\theta_N(x)\int_{B_N}\psi(y)\,d^ny,
\end{equation}
and we continue
\begin{equation}\Label{fin1}\
\eqref{eq:5-11}=c^\pr(t)\theta_N(x)\int_{B_N}\psi(y)\,d^ny-e^{\varkappa \la_N t}W h_2(t,x).
\end{equation}
To finish the proof it remains to show that r.h.s. of \eqref{fin1} equals zero for $x\in B_N$.

Substituting definition \eqref{h2*} of $h_2$ into the expression for $W$, making calculations for $x\in B_N$ and noting that $\psi\in \cD(\Qpn)$, we have:
\begin{align*}
&\big(W h_2\big)(t,x) ={\varkappa} \int_{\Qpn}\dfrac{h_2(x-y)-h_2(x)}{w(\Vert y\Vert_p)}\, d^ny={\varkappa}\int_{\Qpn}\dfrac{h_2(x-y)}{w(\Vert y\Vert_p)}\, d^ny=\\
&={\varkappa}\int\limits_{\Qpn\backslash B_N}\dfrac{h_2(x-y)}{w(\Vert y\Vert_p)}\, d^ny={\varkappa}\int\limits_{\Qpn\backslash B_N}\dfrac{h_2(z)}{w(\Vert x-z\Vert_p)}\, d^nz=\\
&={\varkappa}\int\limits_{\Qpn\backslash B_N}\dfrac{h_2(z)}{w(\Vert z\Vert_p)}\, d^nz.
\end{align*}
Applying further Lemma \ref{5*} and changing the variable on the last step, we have:
\begin{align*}
\big(W h_2\big)(t,x)&=\varkappa\int_{\Qpn \backslash B_N}\dfrac{d^ny}{w(\Vert y\Vert_p)}\int_{B_N}Z(t,y-\eta)\,\psi(\eta)\,d^n\eta=\\
\nonumber
&={\varkappa}\int_{B_N}\psi(\eta)\,d^n\eta \int_{\Qpn\backslash B_N} Z(t,y-\eta) \dfrac{d^ny}{w(\Vert y\Vert_p)}=\\
&=\varkappa\int_{B_N}\psi(\eta)\,d^n\eta\cdot  \int_{\Qpn\backslash B_N} Z(t,\zeta) \dfrac{d^n\zeta}{w(\Vert \zeta\Vert_p)}.
\end{align*}
Thus for $x\in B_N$ \eqref{fin1} looks as follows:
$$
\eqref{fin1}=c^\pr(t)\int_{B_N}\psi(y)\,d^ny+e^{\varkappa \la_N t}\varkappa\int_{B_N}\psi(\eta)\,d^n\eta\cdot  \int_{\Qpn\backslash B_N} Z(t,\zeta) \dfrac{d^n\zeta}{w(\Vert \zeta\Vert_p)},
$$
therefore it remains to show that
\begin{equation}\Label{cpr1}\
c^\pr(t)=-e^{\varkappa \la_N t} \varkappa \int_{\Qpn\backslash B_N}
\dfrac{Z(t,\zeta)}{w(\Vert \zeta\Vert_p)}\,{d^n\zeta}.
\end{equation}
Let us remark that due to \eqref{Zt1}
\begin{align}\Label{fin2}\ \nonumber
&\int\limits_{\Qpn\backslash B_N} Z(t,\zeta)\dfrac{d^n\zeta}{w(\Vert \zeta\Vert_p)}=\int\limits_{\Qpn\backslash B_N} \int\limits_{\Qpn}e^{-\varkappa t A_w(\xi)}\chi(-\zeta\cdot \xi)\,d^n\xi\,\dfrac{d^n\zeta}{w(\Vert \zeta\Vert_p)}
=\\
\nonumber
&= \int\limits_{\Qpn}e^{-\varkappa t A_w(\xi)}\int\limits_{\Qpn\backslash B_N}\chi(-\zeta\cdot \xi)\,\dfrac{d^n\zeta}{w(\Vert \zeta\Vert_p)}\,d^n\xi=\\
&= \int\limits_{\Vert \xi\Vert \leq p^{-N}}e^{-\varkappa t A_w(\xi)}\Big[
-A_w(\xi)+\la_N\Big]\,d^n\xi.
\end{align}
On the last step we used that due to the expression for $\la_N$ \eqref{la-1}, representation of $A_w(\xi)$ \eqref{Aw} and Lemma \ref{Iz} we have
\begin{align*}
\int\limits_{\Qpn\backslash B_N}\chi(-\zeta\cdot \xi)\,\dfrac{d^n\zeta}{w(\Vert \zeta\Vert_p)}&= \int\limits_{\Qpn\backslash B_N}\big[\chi(-\zeta\cdot \xi)-1\big]\,\dfrac{d^n\zeta}{w(\Vert \zeta\Vert_p)}+\la_N=\\
&=-A_w(\xi)+I_{B_N}+\la_N=\left\{
\begin{array}{lcl}
-A_w(\xi)+\la_N&,\ \text{if} &\Vert \xi\Vert_p \leq p^{-N};\\
0&,\ \text{if} &\Vert \xi\Vert_p > p^{-N}.
\end{array}
\right.
\end{align*}
From the other side, due to representation \eqref{Zt1} for $Z(t,x)$, we have the following representation for $c(t)$:
\begin{align}\Label{fin3}\ \nonumber
c(t)&=p^{-nN}-e^{\varkappa\la_N t}p^{-nN}\int_{B_N}Z(t,x)\,d^nx=\\
\nonumber
&=p^{-nN}-e^{\varkappa\la_N t}p^{-nN}\int_{B_N}\int_{\Qpn}e^{-\varkappa t A_w(\xi)}\chi(-x\cdot \xi)\,d^n\xi\,d^nx=\\
&=p^{-nN}-e^{\varkappa\la_N t}\int\limits_{\Vert \xi\Vert_p\leq p^{-N}}e^{-\varkappa t A_w(\xi)}\,d^n\xi.
\end{align}
On the last step we used \eqref{formula11}, see also \cite[(7.14), p. 25]{VTab}, i.e.
  $$\int_{B_N}\chi(\xi\cdot x)d^n x= p^{N n}\left\{
\begin{array}{lcl}
1,&\text{if} &\Vert \xi\Vert_p\leq p^{-N};\\
0,& &\text{otherwise}.
\end{array}
\right.$$
From \eqref{fin3} it follows that
\begin{equation}\Label{fin4}\
c^\pr(t)=-\varkappa \la_N e^{\varkappa \la_N t}
\int_{\Vert \xi\Vert_p\leq p^{-N}}e^{-\varkappa t A_w(\xi)}\,d^n\xi+e^{\varkappa \la_N t}\varkappa \int_{\Vert \xi\Vert_p\leq p^{-N}}A_w(\xi) e^{-\varkappa t A_w(\xi)} d^n\xi
\end{equation}
Noting \eqref{fin4} and \eqref{fin2} we receive the identity \eqref{cpr1}, which proves the statement of the theorem.
\end{proof}

\begin{rem}\rm Let us note that from \eqref{ct} it follows the following representation for $c^\pr(t)$:
\begin{align}\Label{ctpr}\
c^\pr(t)&=-\dfrac{\varkappa \la_N}{\mm(B_N)}\,e^{\varkappa\la_N t}\int\limits_{B_N}Z(t,x)\,d^nx-\dfrac{e^{\varkappa\la_N t}}{\mm(B_N)}\int\limits_{B_N}D_t Z(t,x)\,d^nx.
\end{align}
\end{rem}
\begin{lemma}\Label{cpr}
\begin{equation}\Label{ctpr1}\
c^\pr(t)=\varkappa\, e^{\varkappa \la_N t}
\int_{B_N}e^{-\varkappa t A_w(\xi)}\big[A_w(\xi)- \la_N\big] \,d^n\xi.
\end{equation}
Moreover $c^\pr(0)=0.$
\end{lemma}
\begin{proof} Representation \eqref{ctpr1} follows from the formula \eqref{fin4}. To prove second statement of the theorem, it is sufficient to show that $$\int_{B_N}A_w(\xi)d^n\xi=\int_{B_N} \la_N \,d^n\xi = \,p^{-nN} \,\la_N.$$
Similar to Lemma \ref{lem:4-6} we have
\begin{align*}
&\int\limits_{\Vert z\Vert_p\,\leq\, p^{-N}}A_w(z)\,d^n
=\int\limits_{\Vert y\Vert_p >p^N} \dfrac{d^ny}{w(\Vert y\Vert_p)}
\int\limits_{\Vert z\Vert_p\,\leq\, p^{-N}}\Big(1-\chi(y\cdot z)\Big)\,d^nz=\\
&=\int\limits_{\Vert y\Vert_p >p^N} \dfrac{d^ny}{w(\Vert y\Vert_p)}
\sum\limits_{j=-\infty}^{-N}\,\int\limits_{S_j}\Big(1-\chi(y\cdot z)\Big)\,d^nz=\\
&=\sum\limits_{k=N+1}^\infty\,\int\limits_{S_k} \dfrac{d^ny}{w(\Vert y\Vert_p)}
\sum\limits_{j=-\infty}^{-N}\,\int\limits_{S_j}\Big(1-\chi(y\cdot z)\Big)\,d^nz.
\end{align*}
Let $\Vert y\Vert_p=p^{\,k}$, $k\geq N+1$, then $y=p^{-k}y_0$, where $\Vert y_0\Vert_p=1$. Using change  of variables formula \eqref{change} and \eqref{chi-1} we have
\begin{align}\Label{4-15-1}\
\nonumber
&\sum\limits_{j=-\infty}^{-N}\,\int_{S_j}\Big(1-\chi(y\cdot z)\Big)\,d^nz=\sum\limits_{j=-\infty}^{-N}\,\int_{S_j}\Big(1- \chi(p^{-k}y_0\cdot z)\Big)\,d^nz=\\
\nonumber
&=\sum\limits_{j=-\infty}^{-N}\,p^{-kn}\int_{S_{j+k}}\Big(1-\chi(y_0\cdot v)\Big)\,d^nv=\\
\nonumber
&=\sum\limits_{\ell =-\infty}^{k-N}\,p^{-kn}\int_{S_{\ell}}\Big(1-\chi(y_0\cdot v)\Big)\,d^nv=\\
&=\left\{
\begin{array}{lcl}
\sum\limits_{\ell =-\infty}^{0}\,p^{-kn}\big[(1-p^{-n})p^{n\ell}-p^{n\ell} (1-p^{-n})\big]=0,&\ &\ell \leq\  0;\\
p^{-kn}\big[(1-p^{-n})p^{n}+p^np^{-n}\big]=p^{-kn}p^n,&\ &\ell =1;\\
\sum\limits_{\ell =2}^{k-N}\ p^{-kn}p^{n\ell}(1-p^{-n}),&\ &\ell >1.
\end{array}
\right.
\end{align}
Finally, for $y$ such that $\Vert y\Vert_p=p^k$, $k\geq N+1$, from \eqref{4-15-1} we have
\begin{align*}
&\sum\limits_{j=-\infty}^{-N}\,\int_{S_j}\Big(1- \chi(y\cdot z)\Big)\,d^nz=\sum\limits_{\ell =2}^{k-N}\,p^{-kn}\,p^{\ell n}(1-p^{-n}) \,+\,p^{-kn}p^n=\\
&=
p^{-kn}\,\Big(\sum\limits_{\ell =1}^{k-N}p^{\ell n}-\sum\limits_{\ell =2}^{k-N}p^{\ell n}
p^{-n}\Big)=p^{-kn}\,\Big(\sum\limits_{\ell =1}^{k-N}p^{n\ell}-\sum\limits_{s =1}^{k-N-1}p^{ns}
\Big)=\\
&=p^{-kn}\,p^{n(k-N)}=p^{-nN}.
\end{align*}
Thus
\begin{align*}
&\int\limits_{\Vert z\Vert_p\,\leq\, p^{-N}}A_w(z)\,d^nz=\sum\limits_{k=N+1}^\infty\,\int\limits_{S_k} \dfrac{d^ny}{w(\Vert y\Vert_p)}
\sum\limits_{j=-\infty}^{-N}\,\int\limits_{S_j}\Big(1-\chi(y\cdot z)\Big)\,d^nz=\\
&=\,p^{-nN}\sum\limits_{k=N+1}^\infty\,\int\limits_{S_k} \dfrac{d^ny}{w(\,p^{\,k})}\,
=\,p^{-nN}(1-p^{-n})\sum\limits_{k=N+1}^\infty\,\dfrac{p^{nk}}{w(\,p^{\,k})}=\,p^{-nN} \,\la_N.
\end{align*}
where $\la_N$ is given by \eqref{la}.
\end{proof}
\begin{lemma}The function $Z_N(t.x)$ is non-negative, and
\begin{equation}\Label{Z1}\
\int_{B_N}Z_N(t,x)\, d^nx=1.
\end{equation}
\end{lemma}
\begin{proof}  From \eqref{Zr} and \eqref{ct} we have for $x\in B_N$
\begin{align*}
Z_N(t,x) &= e^{\varkappa\la_N t}Z(t,x) +c(t)=\\
&=e^{\varkappa\la_N t}Z(t,x) +\dfrac{1}{\mm(B_N)}-\dfrac{e^{\varkappa\la_N t}}{\mm(B_N)}\int_{B_N}Z(t,x)\,d^nx=\\
&=e^{\varkappa\la_N t}\Big[Z(t,x) -p^{-nN}\int_{B_N}Z(t,x)\,d^nx\Big]+p^{-nN}=1.
\end{align*}
The positivity of the function $Z_N(t,x)$ follows from its probabilistic meaning as the transition density of the process $\eta_t$ \eqref{etat}.
\end{proof}
On a ball $B_N$, $N\in \ZZ$ let us consider the Cauchy problem \eqref{Cauchy}. Its fundamental solution $Z_N(t,x)$ \eqref{Zr} defines a contraction semigroup
\[(T_N(t)u)(x)=\int_{B_N} Z_N(t, x-\xi)\, u(\xi)\, d^n\xi\]
on $L^1(B_N)$.
\begin{lemma}\Label{l5-4}\
The semigroup $T_N(t)$ is strongly continuous in $L^1(B_N)$.
\end{lemma}
\begin{proof} For $u\in L^1(B_N)$ we may write $\Vert T_N(t)u-u\Vert_{L^1(B_N)}\leq I_1(t)+I_2(t),$
where
\begin{align*}
I_1(t)&=\int_{B_N} \Bigg\vert \int_{B_N}e^{\varkappa \la_N t}Z(t,x-\xi)\,u(\xi)\,d\xi - u(x)\Bigg\vert\, d^nx;\\
I_2(t)&= p^{nN}\,\vert c(t)\vert \int_{B_N} \vert u(\xi)\vert \,d^n\xi.
\end{align*}
Using representation \eqref{Zt1} and \eqref{formula11} from \eqref{ct} we have
\begin{align}\ \Label{gtt0}\ \nonumber
c(t)&=p^{-nN}-e^{\varkappa\,\la_N t}p^{-nN}\int_{\Qpn}e^{-\varkappa t A_w(\xi)}\int_{B_N}\chi(-x\cdot \xi)\,d^nx\,d^n\xi=\\
&=p^{-nN}-e^{\varkappa\la_N t}\int\limits_{\Vert \xi\Vert_p\leq p^{-N}}e^{-\varkappa t A_w(\xi)}\,d^n\xi \to 0,\ \ \text{\rm as}\ \ t\to 0.
\end{align}
Therefore $I_2(t)\to 0$ as $t\to 0$.
For small values of $t$ we write
\begin{align*}
I_1(t)&=\int\limits_{B_N}\Bigg\vert\int\limits_{B_N}Z(t,x-\xi)u(\xi)\, d^n\xi - u(x)+\int\limits_{B_N}\big(e^{\varkappa \la_N t}-1\big)Z(t,x-\xi)u(\xi)\,d^n\xi \Bigg\vert\, d^nx\leq\\
&\leq \int\limits_{B_N}\Bigg\vert\int\limits_{B_N}Z(t,x-\xi)\,u(\xi)\, d^n\xi - u(x)\Bigg\vert\,d^nx+Ct\int\limits_{B_N}\int\limits_{B_N}Z(t,x-\xi)\vert u(\xi) \vert\, d^n\xi\,d^nx=\\
&=J_1(t)+J_2(t).
\end{align*}
By the Young inequality using the identity \eqref{Zt5}, extending $u$ by zero to a function $\w u$ on $\Qpn$, we obtain
\[
J_2(t)\leq Ct \int_{\Qpn}\int_{\Qpn}Z(t,x-\xi)\,\vert \w u(\xi) \vert\, d^n\xi\,d^nx\leq Ct \Vert \w u\Vert_{L^1(B_N)}\to 0,\ \ \text{\rm as}\ \ t\to 0.
\]
Moreover by the $C_0$-property of $T(t)$ stated in Lemma \ref{l3-1} we have
\begin{align*}
J_1(t)&= \int_{B_N}\Bigg\vert\,\int_{\Qpn}Z(t,x-\xi)\,\w u(\xi)\, d^n\xi - \w u(x)\Bigg\vert\,d^nx\leq\\
&\leq \int\limits_{\Qpn}\Bigg\vert\,\int\limits_{\Qpn}Z(t,x-\xi)\,\w u(\xi)\, d^n\xi - \w u(x)\Bigg\vert\,d^nx=\Vert T(t)\w u- \w u\Vert_{L^1(\Qpn)}\to 0,
\end{align*}
as $t\to 0$.
\end{proof}
Let $\fA_N$ denote the generator of the contraction semigroup $T_N(t)$ in $L^1(B_N)$. Let us also introduce operator $W_N$ which is understood in the sense of $\cD^\pr(B_N)$, that is $\psi_N$ is extended by zero to a function on $\Qpn$, $W_N$ is applied to it in the distribution sense, and the resulting distribution is restricted to $B_N$.
\begin{prop}\Label{prop5-25-1}\ Let operator $\fA$ be a generator of semigroup $T(t)$ in space $L^1(\Qpn)$ with the domain $Dom (\fA)$. Then for the restriction $\psi_N$ of the function $\psi\in Dom (\fA)$ to the ball $B_N$ we have:
\begin{equation}\Label{5-25-1}\
\fA\psi = W_N\psi_N + R_N,
\end{equation}
where $R_N = R_N (\psi - \psi_N)$ is the constant from Lemma \ref{5*}.
\end{prop}
\begin{proof} To prove this statement we remark that function $\psi \in Dom\, (\fA)$ may be represent as $\psi = \psi_N+(\psi-\psi_N)$ and by Lemma \ref{5*}  on $B_N$ we may write $
\fA\psi = W_N\psi_N + R_N, \ \text{where}\ R_N = R_N (\psi - \psi_N).$
\end{proof}
\begin{theorem} If $\psi \in Dom \,(\fA)$ in $L^1(\Qpn)$ (Definition \ref{def:3-2}), then the restriction $\psi_N$ of the function $\psi$ to $B_N$ belongs to $Dom \,(\fA_N)$ and
\begin{equation}\Label{Th5-5}\
\fA_N\psi_N=\big(  W_N-\varkappa \la_N\big) \psi_N,
\end{equation}
where $W_N \psi_N$ is understood in the sense of $\cD^\pr(B_N)$, that is $\psi_N$ is extended by zero to a function on $\Qpn$, $W_N$ is applied to it in the distribution sense, and the resulting distribution is restricted to $B_N$.
\end{theorem}
\begin{proof}
For $\psi \in Dom \,(\fA)$ we have to check that:
\begin{itemize}
\item[1)] $W_N\psi_N\in L^1(B_N)$;
\item[2)] $\Big\Vert -\dfrac{1}{t}\big[T_N(t)\psi_N -\psi_N\big]-\big(W_N-\varkappa \la_N\big)\psi_N\Big\Vert_{L^1(B_N)}\to 0$, as $t\to 0+$.
\end{itemize}
To prove the first statement we remark that for function $\psi \in Dom\, (\fA)$ Proposition \ref{prop5-25-1} implies that $\fA\psi = W_N\psi_N + R_N$ thus $W_N\psi_N \in L^1(B_N)$.

Further, from \eqref{Zr}, expanding exponent in the Taylor series, we have
\begin{align}\Label{5-25}\
\big(T_N(t) \psi_N\big)(x)&=\int_{B_N}Z(t,x-y)\psi(y)\, d^n y+c(t)\int_{B_N}\psi(y)\, d^ny +\\
\nonumber
&+\varkappa \,\la_N t \int_{B_N} Z(t,x-y)\psi(y)\, d^n y+d(t)\int_{B_N} Z(t,x-y)\psi(y)\,d^n y,
\end{align}
where $d(t)=O(t^2)$, $t\to 0$. By strong continuity of $T_N(t)$ in $L^1(B_N)$ (see Lemma \ref{l5-4}) we have
\[\Bigg\Vert -\varkappa \la_N \int\limits_{B_N}Z(t,x-y)\psi(y)\, d^n y +\varkappa\la_N \,\psi_N(x)\Bigg\Vert_{L^1(B_N)}\to 0,\ \ \text{\rm as}\ \ t\to 0.\]
Moreover from Young inequality it follows that
\[\dfrac{1}{t}\Bigg\Vert d(t)\int\limits_{B_N}Z(t,x-y)\psi(y)\,d^ny\,\Bigg\Vert_{L^1(B_N)}\to 0,\ \ \text{\rm as}\ \ t\to 0.\]

From \eqref{gtt0} it follows that $c(0)=0$. Moreover Lemma \ref{cpr}  implies that $c^\pr(0)=0$, thus $c(t)=O(t^2)$ as $t\to 0$ and second term in \eqref{5-25} is negligible. Therefore it remains to consider the first term in \eqref{5-25}, that is
\[V(t,x)=\int_{B_N}Z(t, x-y)\,\psi(y)\, d^n y = v_1(t,x)-v_2(t,x),\ \ x\in B_r,\]
where
\begin{align*}
&v_1(t,x)=\int_{\Qpn}Z(t,x-y)\,\psi(y)\,d^ny,\\
&v_2(t,x)=\int\limits_{\Vert y\Vert_p >p^N}Z(t,x-y)\,\psi(y)\,d^ny.
\end{align*}

If we show that
\begin{equation}\Label{v2}\
\dfrac{1}{t} v_2(t,x) - R_N \to 0,\ \ \text{when}\ \ t\to 0
\end{equation}
then this ends the proof. Indeed, taking above argument and \eqref{5-25-1} in account we may write:
\begin{align*}
&\Big\Vert -\dfrac{1}{t}\big[T_N(t)\psi_N -\psi_N\big]-\big(W_N-\varkappa \la_N\big)\psi_N\Big\Vert_{L^1(B_N)}\leq\\
&\leq \Big \Vert -\dfrac{1}{t} \Big( \int_{\Qpn}Z(t,x-y)\, \psi (y)\, d^ny - v_2(t,x)-\psi\Big)-W_N\psi_N \Big\Vert_{L^1(B_N)} + o(1) =\\
&= \Big \Vert -\dfrac{1}{t} \Big( \int_{\Qpn}Z(t,x-y)\, \psi(y)\, d^ny -\psi\Big)-(W_N\psi_N +R_N) \Big\Vert_{L^1(B_N)} + o(1) =\\
&= \Big \Vert -\dfrac{1}{t} \big( T(t)\psi -\psi\big)-W\psi \Big\Vert_{L^1(B_N)} + o(1) \to 0,\ \ \text{as}\ \ t\to 0,
\end{align*}
which completes the proof.

Let us show \eqref{v2}. To do this we write
\begin{equation}\Label{5-28}\
\dfrac{1}{t}v_2(t,x)-R_N =\int_{\Vert y\Vert_p >p^N}\Big(\dfrac{1}{t}Z(t,y)-\dfrac{\are}{w(\Vert y\Vert_p)}\Big)\psi(y) \,d^ny,
\end{equation}
where $\psi\in Dom\, (\fA)$ in $L^1(\Qpn)$.
If we show that
\begin{equation}\Label{Ztare}\
Z(t,y) = \dfrac{\are t}{w(\Vert y\Vert_p)} +o(t^2),
\end{equation}
then this will prove \eqref{v2}.

Due to \eqref{Zt7} for $\Vert y\Vert_p=p^\be$ we have
\[Z(t,y) =\Vert y\Vert_p^{-n}\Bigg[(1-p^{-n})\sum\limits_{j=0}^\infty p^{-nj}e^{-\are t A_w(p^{-(\be+j)})}-e^{-\are tA_w(p^{-(\be -1)})}\Bigg].\]
Then expanding the exponents in the Taylor series we have
\begin{align*}
&e^{-\are t A_w(p^{-(\be+j)})}=1-\are t A_w(p^{-(\be+j)})+r_1(t);\\
&e^{-\are t A_w(p^{-(\be-1)})}=1-\are t A_w(p^{-(\be-1)})+r_2(t),
\end{align*}
where, for example, $r_2(t)$ is the result of decomposition of function $f(t)=e^{-\are tA_w(\Vert y\Vert_p^{-1} p)}$:
\begin{align*}
&r_2 (t)= \dfrac{t^2f^{\pr\pr}(\theta t)}{2}= \big[A_w(\Vert y\Vert_p^{-1}p)\big]^2e^{-\are \theta t A_w(\Vert y\Vert_p^{-1} p)}, \ \theta \in (0,1).
\end{align*}
Thus, due to \eqref{Aw1}
\[\vert r_2(t)\leq C\,t^2\big[\Vert y\Vert^{-1}_pp\big]^{2(\al - n)},\ \ \text{for}\ \Vert y\Vert_p=p^\be.\]
Therefore $r_2=o(t^2)$ as $t\to 0$ and similar $r_1(t)$, thus
\begin{align*}
Z(t,y)&=\Vert y\Vert_p^{-n}\Bigg[(1-p^{-n})\sum\limits_{j=0}^\infty p^{-nj}\big(1-\are t A_w(p^{-(\be+j)})\big)-1+\are t A_w(p^{-(\be-1)})\Bigg]+ o(t^2)=\\
&=\Vert y\Vert_p^{-n}\are\,t\Big[A_w(p^{-(\be-1)})-(1-p^{-n})\sum\limits_{j=0}^\infty p^{-nj} A_w(p^{-(\be+j)})\Big]+ o(t^2).
\end{align*}
Using representation \eqref{eq:Awrep-1} we may write
\begin{align}\Label{5-30}\
Z(t,y)&=\Vert y\Vert_p^{-n}\are\,t\Big[A+B-C-D\Big]+ o(t^2),\\
\intertext{where}
\nonumber
A&=(1-p^{-n})\sum\limits_{k=\be+1}^\infty\dfrac{p^{nk}}{w(p^{\,k})};\quad B=\dfrac{p^{n\be}}{w(p^\be)};\\
\nonumber
C&= (1-p^{-n})\sum\limits_{j=0}^\infty p^{-nj}(1-p^{-n}) \sum\limits_{k=\be+j+2}\dfrac{p^{nk}}{w(p^{\,k})};\\
\nonumber
D&= (1-p^{-n})\sum\limits_{j=0}^\infty p^{-nj}\dfrac{p^{n(\be+j+1)}}{w(p^{\be+j+1})}.
\end{align}
Let us in the term $C$ change the order of summation and calculate the finite sum of geometric progression, then we have
\begin{align*}
C&=(1-p^{-n})^2\sum\limits_{k=\be+2}^\infty\sum\limits_{j=0}^{k-\be-2}p^{-nj}\dfrac{p^{nk}}{w(p^{\,k})}=(1-p^{-n})\sum\limits_{k=\be+2}^\infty\dfrac{p^{nk}}{w(p^{\,k})}\big(1-p^{-n(k-\be-1)}\big)=\\
&=(1-p^{-n})\sum\limits_{k=\be+2}^\infty\dfrac{p^{nk}}{w(p^{\,k})}-(1-p^{-n})p^{n(\be+1)}\sum\limits_{k=\be+2}^\infty\dfrac{1}{w(p^{\, k})}=C_1-C_2;\\
D&=(1-p^n)p^{n(\be+1)}\sum\limits_{j=0}^\infty\dfrac{1}{w(p^{\be+j+1})}=(1-p^n)p^{n(\be+1)}\sum\limits_{k=\be+1}^\infty\dfrac{1}{w(p^{\, k})}.
\end{align*}
Looking at \eqref{5-30} we see that
\begin{align*}
&A-C_1 = (1-p^{-n})\dfrac{p^{n(\be+1)}}{w(p^{\,\be+1})};\\
&C_2-D = -(1-p^{-n})\dfrac{p^{n(\be+1)}}{w(p^{\,\be+1})},
\end{align*}
which contracts each other and it remains the term $B$, i.e.
\[Z(t,y)=\Vert y\Vert_p^{-n}\are\,t\dfrac{p^{n\be}}{w(p^{\be})}+ o(t^2)=\dfrac{\are t}{w(\Vert y\Vert_p)}+ o(t^2),\ \text{for}\ \Vert y\Vert_p=p^\be,\]
which proves \eqref{Ztare} and therefore \eqref{v2}.
\end{proof}
\section{Main result}

To formulate the main result let us first recall the notation of mild solution of nonlinear equation in some real Banach space $X$.

\subsection{Mild solution}

Consider the Cauchy problem
\begin{equation}\Label{4-1Barbu}
\left\{
\begin{array}{lc}
D_tu+Au = f(t),&t\in[0,T];\\
u(0)=u_0,&
\end{array}
\right.
\end{equation}
where $u_0\in X$ and $f\in L^1([0,T]; X)$, $A$ is a $m$-accretive nonlinear operator.

Operator $A\colon X\to  X$ is called \bfi{accretive} if for every pair $x,y\in Dom\, (A)$
\[\langle Ax-Ay, w\rangle \geq 0,\]
where $w\in J(x-y)$ and $J\colon X\to X^*$ is the duality mapping of the space $X$. Cor\-res\-pon\-dingly operator $A$ is called \bfi{$m$-accretive} if the range $Ran\,(I+A)=X$.
\begin{definition} Let $f\in L^1([0;T];X)$ and $\vep > 0$ be given. An \bfi{$\vep$-discretization} on $[0;T]$ of the equation $D_t y +Ay = f$ consists of a partition $0 = t_0\leq t_1 \leq t_2 \leq \ldots\leq t_N$ of the
interval $[0; t_N]$ and a finite sequence $\{f_i\}_{i=1}^N\subset X$
such that $
t_i-t_{i-1}<\vep\ \ \text{for}\ \ i=1,\ldots, N,\ \ T-\vep < t_N \leq T
$ and
\[\sum\limits_{i=1}^N \,\int\limits_{t_{i-1}}^{t_i}\Vert f(s)-f_i\Vert \, ds <\vep.\]
\end{definition}
\begin{definition} A piecewise constant function $z\col [0,t_N]\to X$ whose values $z_i$ on $(t_{i-1}, t_i]$ satisfy the finite difference equation
$$
\dfrac{z_i-z_{i-1}}{t_i-t_{i-1}}+Az_i=f_i,\ \ \ i=1,\ldots, N
$$
is called an \bfi{$\vep-$approximate solution} to the Cauchy problem \eqref{4-1Barbu} if it satisfies
$$
\Vert z(0)-u_0\Vert \leq \vep.
$$
\end{definition}
\begin{definition}
\bfi{A mild solution} of  the Cauchy problem \eqref{4-1Barbu} is a function $u\in C([0,T]; X)$ with the property that for each $\vep >0$ there us an \bfi{$\vep-$approximate solution} $z$ of $D_tu +Au = f$ on $[0,T]$ such that $\Vert u(t)- z(t)\Vert \leq \vep $ for al $t\in [0,T]$ and $u(0)=u_0$.
\end{definition}
See \cite[Ch. 4]{Barbu:book} for the details.

\subsection{Solvability of the nonlinear equation}

Let us consider in $L^1(\Qpn)$ the equation
\begin{equation}\Label{6-1}\
D_t u+\fA \big(\vph(u)\big)=0,\ \ u=u(t,x),\ t>0,\ x\in\Qpn,
\end{equation}
where $\fA$ is the generator of the semigroup $T(t)$ in $L^1(\Qpn)$ and $\vph\col \RR\to \RR$ is a continuous strictly increasing function, $\vph(0)=0$, such that:
$$
\vert \vph (s)\vert \leq C\, \vert s\vert^m,\ \ m\geq 1.
$$
Consider the nonlinear operator $\fA\vph$ with the domain
$$
Dom\, (\fA\vph) =\{u\in L^1(\OO)\col \vph(u)\in  Dom\,(\fA)\}.
$$
From Lemma \ref{lem3-2} it follows that the operator $\fA \vph$ is densely defined and therefore so the operator $\w{\fA \vph}$ has the same property.

\begin{theorem} The operator $\overline{\fA \vph}$ is $m$-accretive, i.e. for any initial function $u_0\in L^1(\Qpn)$ the Cauchy problem for equation \eqref{6-1} has a unique mild solution.
\end{theorem}
\begin{proof} The statement of the theorem is a consequence of the Crandall-Liggett theorem \cite[Theorem 4.3]{Barbu:book}. Indeed,
from Proposition 1 in \cite{CrP} it follows that $(\fA\vph)(u)=\fA(\vph(u))$ is an accretive nonlinear operator in $L^1(\Qpn)$ and for any $\vep >0$ operator $(\vep I+\fA)\vph$ is $m$-accretive in $L^1(\Qpn)$.

Therefore in order to prove the $m$-accretivity of $\overline{\fA \vph}$ it is sufficient to prove that the operator $I+\fA\vph$ has a dense range in $L^1(\Qpn)$. In other words it suffices to prove that equation
\[u+\fA\vph(u)=f\]
is solvable for a dense subset of functions $f\in L^1(\Qpn)$.  Equivalently, setting  $\be = \vph^{-1}$ (the function inverse to $\vph$), we have to study the equation
\begin{equation}\Label{v}\
\fA v+\be(v)=f.
\end{equation}
Since the space of test functions $\cD(\Qpn)$ is dense in $L^1(\Qpn)$, therefore it is enough to prove the solvability of equation \eqref{v} for any $f\in L^1(\Qpn)\cap L^\infty(\Qpn)$.

For such a function $f$ we consider the regularized equation to \eqref{v}:
\begin{equation}\Label{ve}\
\vep v_\vep +\fA v_\vep+\be(v_\vep)=f, \quad \vep >0,
\end{equation}
possessing, due to Proposition 4 in \cite[p.\,571]{BrSt:1973} a unique solution $v_\vep$, such that $w_\vep = f-\fA v_\vep$ satisfies the inequality:
\begin{equation}\Label{2-17K}\
\Vert w_\vep\Vert_{L^1(\Qpn)}\leq \Vert f\Vert_{L^1(\Qpn)}.
\end{equation}
Moreover, if $\w{v}_\vep$ and $\w{w}_\vep$ correspond to equation \eqref{ve} with r.h.s. $f$ and $\w{f}$, then
\begin{equation}\Label{2-18K}\
\Vert w_{\vep}-\w{w}_\vep\Vert_{L^1(\Qpn)}\leq\Vert f - \w f\,\Vert_{L^1(\Qpn)}.
\end{equation}
In addition, if $f\in L^1(\Qpn)\cap L^\infty(\Qpn)$, then due to the Proposition 4 in \cite{BrSt:1973} applied in space $L^q(\Qpn)=L^\infty(\Qpn)$ we have
\begin{equation}\Label{2-19K}\
\Vert f - (\fA+\vep)v_\vep\Vert_{L^\infty(\Qpn)}=\Vert \be(v_\vep)\Vert_{L^\infty(\Qpn)}\leq \Vert f\Vert_{L^\infty(\Qpn)}.
\end{equation}
Using inequality \eqref{2-19K} we find that
\begin{equation}\Label{5-2}\
\vert v_\vep(x)\vert \leq \be^{-1}(\Vert f\Vert_{L^\infty(\Qpn)})
\end{equation}
for almost all $x\in\Qpn$. This means that for any fixed $N$ the constant $R_N(v_\vep)$ from \eqref{Rr} satisfy inequality:
\[R_N(v_\vep)\leq C,\]
where $C$ does not depend on $\vep$, so that the set of constant functions $\{R_N(v_\vep), 0 < \vep < 1\}$ is relatively compact in $L^1(B_N)$.

On the other hand, it follows from \eqref{2-17K}, \eqref{2-18K} and the translation invariance of $\fA$ that the family of functions $w_\vep = f-\fA v_\vep$ satisfies the inequalities:
\begin{align}\Label{5-3K}\
&\Vert w_\vep\Vert_{L^1(\Qpn)}\leq \Vert f\Vert_{L^1(\Qpn)};\\
\Label{5-4K}\
&\int\limits_{\Qpn}\vert w_\vep(x+h)-w_\vep(x)\vert\, d^nx\leq\int\limits_{\Qpn}\vert f(x+h) -f(x)\vert\, d^nx
\end{align}
for any $h\in\Qpn$.
The conditions \eqref{5-3K} and \eqref{5-4K} imply relative compactness of sequence $\{w_\vep\}$ and therefore of $\{\fA v_\vep\}$ in $L^1_{\text{loc}}(\Qpn)$, that is the compactness of the closure of the restriction $(\fA v_\vep)\big\vert_X$ for any bounded measurable subset $X\subset \Qpn$. This is a consequence of the criterion for relative compactness in $L^1(G)$ where $G$ is a compact group (see Theorem 4.20.1 in \cite{Edwards}) applied to the case $G=B_N$ (the additive group of $p$-adic ball).

Denote by $v_{\vep,N}$ the restriction of $v_\vep$ to $B_N$. From Proposition \ref{prop5-25-1} it follows that 
\[W_N\psi_N = \fA\psi - R_N\]
therefore the set $W_N v_{\vep,N}$ is  relatively compact in $L^1(B_N)$. Since $W_N =\fA_N +\are \la_N$, defined as in \eqref{Th5-5}, due to Hille-Yosida theorem has bounded inverse on $L^1_{loc}(\Qpn)$, this implies the relative compactness of $\{v_{\vep,N}\}$ in $L^1(B_N)$ for each $N$. The same is true for $\{v_\vep\}$ in $L^1_{loc}(\Qpn)$. Let $v$ be its limit point. Together with the relative compactness of $\{\fA v_\vep\}$ the above reasoning proves the solvability of \eqref{v} because by Fatou's lemma and \eqref{5-3K}, a limit point of $\{\fA v_\vep\}$ belongs to $L^1(\Qpn)$. Therefore $\be(v)\in L^1(\Qpn)$. By \eqref{5-2}, $v\in L^\infty (\Qpn)$, so that $\be(v)\in L^\infty(\Qpn)$, $v=\vph(\be(v))$, $\vert v(x)\vert \leq C\, \vert \be (v)\vert$, and $v$ belongs to $L^1(\Qpn).$

\end{proof}

\section*{Acknowledgments}
The work by the first- and third-named authors was funded in part under the budget program of Ukraine No. 6541230 ``Support to the development of priority research trends''. The third-named author was also supported in part in the framework of the research work "Markov evolutions in real and p-adic spaces" of the Dragomanov National Pedagogical  University of Ukraine.

\end{document}